\documentclass{article}

\usepackage{amsmath, amssymb, amsthm, amsfonts, bm}
\usepackage{enumerate}

\numberwithin{equation}{section}
\newtheorem{theorem}{Theorem}[section]
\newtheorem{cthm}{Theorem}
\newenvironment{ctheorem}[1]
  {\renewcommand\thecthm{#1}\cthm}
  {\endcthm}
\newtheorem{corollary}{Corollary}[section]
\newtheorem{lemma}{Lemma}[section]
\theoremstyle{definition}
\newtheorem{definition}{Definition}[section]
\theoremstyle{remark}
\newtheorem{remark}{Remark}[section]
\newtheorem*{cremark}{Remark}

\newcommand{\Rnum}[1]{\uppercase\expandafter{\romannumeral #1\relax}}
\newcommand{\rnum}[1]{\romannumeral #1\relax}
\newcommand{\mr}[1]{\mathrm{#1}}
\newcommand{\mb}[1]{\mathbb{#1}}
\newcommand{\mc}[1]{\mathcal{#1}}

\DeclareMathOperator{\Av}{Av} 

\def\clap#1{\hbox to 0pt{\hss#1\hss}}
\def\mathclap{\mathpalette\mathclapinternal}
\def\mathclapinternal#1#2{\clap{$\mathsurround=0pt#1{#2}$}}

\title{Extremal sequences for the Bellman function of three variables of the dyadic maximal operator related to Kolmogorov's inequality}
\author{Eleftherios N. Nikolidakis}
\date{\today}

\begin{document}
\maketitle

\begin{abstract}
We give a characterization of the extremal sequences for the Bellman function of three variables of the dyadic maximal operator in relation to Kolmogorov's inequality. In fact we prove that they behave approximately like eigenfunctions of this operator for a specific eigenvalue. For this approach we use the one introduced in \cite{11}, where the respective Bellman function has been precisely evaluated.
\end{abstract}

\section{Introduction} \label{sec:1}
The dyadic maximal operator on $\mb R^n$ is a useful tool in analysis and is defined by
\begin{equation} \label{eq:1p1}
\mc M_d\phi(x) = \sup\left\{ \frac{1}{|S|} \int_S |\phi(u)|\,\mr du: x\in S,\ S\subseteq \mb R^n\ \text{is a dyadic cube} \right\},
\end{equation}
for every $\phi\in L^1_\text{loc}(\mb R^n)$, where $|\cdot|$ denotes the Lebesgue measure on $\mb R^n$, and the dyadic cubes are those formed by the grids $2^{-N}\mb Z^n$, for $N=0, 1, 2, \ldots$.\\
It is well known that it satisfies the following weak type (1,1) inequality
\begin{equation} \label{eq:1p2}
\left|\left\{ x\in\mb R^n: \mc M_d\phi(x) > \lambda \right\}\right| \leq \frac{1}{\lambda} \int_{\left\{\mc M_d\phi > \lambda\right\}} |\phi(u)|\,\mr du,
\end{equation}
for every $\phi\in L^1(\mb R^n)$, and every $\lambda>0$,
from which follows in view of Kolmogorov's inequality the following $L^q$-inequality
\begin{equation} \label{eq:1p3}
\int_E \left|\mc M_d\phi(u)\right|^q\mr du \leq \frac{1}{1-q} |E|^{1-q} \|\phi\|_1^q,
\end{equation}
for every $q\in(0,1)$, every $\phi\in L^1(\mb R^n)$ and every measurable subset of $\mb R^n$, $E$, of finite measure. It is not difficult to see that the weak type inequality \eqref{eq:1p2} is best possible. For refinements of this inequality one can see \cite{15}, \cite{17} and \cite{18}.

An approach for studying in more depth the behaviour of this maximal operator is the introduction of the so called Bellman functions related to them which reflect certain deeper properties of them by localizing. Such functions related to the $L^q$ inequality \eqref{eq:1p3} have been precisely evaluated in \cite{11}. Define $\Av_E(\psi)=\frac{1}{|E|} \int_E |\psi|$, where $E\subseteq \mb R^n$ is measurable of positive measure and $\psi$ is measurable on $E$, and fixing a dyadic cube define the localized maximal operator $\mc M'_d\phi$ as in \eqref{eq:1p1} but with the dyadic cubes $S$ being assumed to be contained in the ambient dyadic cube $Q$. Then for every $q\in(0,1)$ we let
\begin{equation} \label{eq:1p4}
B_q(f,h)=\sup\left\{ \frac{1}{|Q|} \int_Q (\mc M'_d\phi)^q: \Av_Q(\phi)=f,\ \Av_Q(\phi^q)=h, \right\}
\end{equation}
where $\phi$ is nonnegative in $L^1(Q)$ and the variables $f, h$ satisfy $0<h\leq f^q$. By a scaling argument it is easy to see that the above is independent of the choice of $Q$ (so we just have written $B_q(f,h)$ and we may take $Q=[0,1]^n$).
In \cite{11}, now the function \eqref{eq:1p4} has been precisely evaluated. The proof has been given in a much more general setting of tree-like structures on probability spaces.
More precisely we consider a non-atomic probability space $(X,\mu)$ and let $\mc T$ be a family of measurable subsets of $X$, that has a tree-like structure similar to the one in the dyadic case (the exact definition will be given in Section \ref{sec:2}).
Then we define the dyadic maximal operator associated with $\mc T$, by
\begin{equation} \label{eq:1p5}
\mc M_{\mc T}\phi(x) = \sup \left\{ \frac{1}{\mu(I)} \int_I |\phi|\,\mr d\mu: x\in I\in \mc T \right\},
\end{equation}
for every $\phi\in L^1(X,\mu)$. \\
This operator is related to the theory of martingales and satisfies essentially the same inequalities as $\mc M'_d$ does. Now we define the corresponding Bellman function of $\mc M_{\mc T}$, by
\begin{multline} \label{eq:1p6}
B_q^Q(f,h,L,k) = \sup \left\{ \int_E \left[ \max(\mc M_{\mc T}\phi, L)\right]^q\mr d\mu: \phi\geq 0, \int_X\phi\,\mr d\mu=f, \right. \\  \left. \int_X\phi^q\,\mr d\mu = h,\ E\subseteq X\ \text{measurable with}\ \mu(E)=k\right\},
\end{multline}
the variables $f, h, L, k$ satisfying $0<h\leq f^q$, $L\geq f$ and $k\in (0,1]$.
The evaluation of \eqref{eq:1p6} is now given in \cite{11}, and has been done in several steps. The first one is to find the value of
\begin{equation} \label{eq:1p7}
B_q^{\mc T}(f,h,f,1) = \sup\left\{ \int_X (\mc M_{\mc T}\phi)^q\,\mr d\mu:\ \phi \geq 0,\ \int_X \phi\,\mr d\mu = f,\ \int_X \phi^q\,\mr d\mu = h\right\}.
\end{equation}
It is proved in \cite{11}, that \eqref{eq:1p7} equals $h\,\omega_q\!\left(\frac{f^q}{h}\right)$ where $\omega_q: [1,+\infty) \to [1,+\infty)$ is defined as $\omega_q(z) = \left[H^{-1}_q(z)\right]^q$, where $H^{-1}_q$ is the inverse of $H_q$ given by $H_q(z) = (1-q)z^q + qz^{q-1}$, for $z\geq 1$. \\
The second step for the evaluation of \eqref{eq:1p6} is to find $B_q^{\mc T}(f,h,L,1)$ for arbitrary $L\ge f$.
We state the related result:
\begin{ctheorem}{1} \label{thm:1}
With the above notation
\begin{equation} \label{eq:1p8}
B_q^{\mc T}(f,h,L,1) = h\,\omega_q\!\left( \frac{(1-q)L^q + qL^{q-1}f}{h}\right).
\end{equation}
\end{ctheorem}
Our aim in this paper is to characterize the extremal sequences of functions involving \eqref{eq:1p8}. More precisely we will prove the following
\begin{ctheorem}{A} \label{thm:a}
Let $\phi_n: (X,\mu) \to \mb R^+$ be such that $\int_X \phi_n\mr d\mu=f$ and $\int_X \phi_n^q\mr d\mu=h$, where $f,h$ are fixed with $0< h\leq f^q$, $q\in (0,1)$ and $n\in \mb N$. Suppose additionally that $L\geq f$. Then the following are equivalent:
\begin{enumerate}[i)]
\item \quad  $\displaystyle \lim_n \int_X \left[ \max(\mc M_{\mc T}\phi_n, L)\right]^q\mr d\mu = B_q^{\mc T} (f, h, L, 1) $
\item \quad $\displaystyle \lim_n \int_X \left| \max(\mc M_{\mc T}\phi_n, L) - c^\frac{1}{q}\phi_n\right|^q\mr d\mu = 0 $,
\end{enumerate}
where $c = \omega_q\!\left(\frac{(1-q)L^q + qL^{q-1}f}{h}\right)$.
\end{ctheorem}
We discuss now the method of the proof of Theorem \ref{thm:a}. We begin by proving two Theorems (\ref{thm:4p1} and \ref{thm:4p2}) which are generalizations of the results in \cite{11}. By using these theorems, we prove Theorem \ref{thm:4p3} which is valid for any extremal sequence and in fact is a weak form of Theorem \ref{thm:a}. We then apply Theorem \ref{thm:4p3} to a new sequence of functions, called $(g_{\phi_n})_n$, which arises from $(\phi_n)_n$ by a natural, as we shall see, way.
The function $g_{\phi_n}$ is in fact equal to $\phi_n$ on the set $\left\{ \mc M_{\mc T} \phi_n \leq L\right\}$, and constant on certain subsets of $E_n = \left\{\mc M_{\mc T}\phi_n > L\right\}$, which are enough for one to describe the behavior of $\mc M_{\mc T}\phi_n$ in $E_n$.
This new sequence has the property that it is in fact arbitrary close to $(\phi_n)_n$, thus it is extremal. An application then of Theorem \ref{thm:4p3} to this new sequence and some combinations of lemmas will enable us to provide the proof of Theorem \ref{thm:a}.

We need also to mention that the extremizers for the standard Bellman function for the case $p>1$ has been studied in \cite{16}, inspired by \cite{10}. In this paper we study the more general case (for the Bellman function of three variables and for $q\in (0,1)$) which presents additional difficulties because of the presence of the third variable $L$.

We note also that further study of the dyadic maximal operator can be seen in \cite{19} and \cite{18} where symmetrizations principles for this operator are presented, while other approaches for the determination of certain Bellman function can be found in \cite{26}, \cite{27}, \cite{31}, \cite{32},  and \cite{33}.

Also we need to say that the phenomenon that the norm of a maximal operator is attained by a sequence of eigenfuntions doesn't occur here for the first time, for example see \cite{4} and \cite{5}.
Nevertheless as far as we know this phenomenon is presented here and in \cite{6} for the first time for the case of more generalized norms, such as the Bellman functions that we describe.

There are several problems in Harmonic Analysis where Bellman functions naturally arise. Such problems (including the dyadic Carleson imbedding and
weighted inequalities) are described in \cite{14} (see also \cite{12}, \cite{13}) and also connections to Stochastic Optimal Control are provided,
from which it follows that the corresponding Bellman functions satisfy certain nonlinear second order PDE.

The exact computation of a Bellman function is a difficult task which is connected with the deeper structure of the corresponding Harmonic Analysis
problem. Thus far several Bellman functions have been computed (see \cite{2}, \cite{3}, \cite{10}, \cite{25}, \cite{27}, \cite{31}, \cite{32}, \cite{33}). L.Slavin, A.Stokolos and V. Vasyunin \cite{26} linked the Bellman function computation to solving certain PDE's of the Monge-Amp\`{e}re type, and in this way they obtained an alternative proof of the Bellman functions related to the dyadic maximal operator in \cite{10}. In this last mentioned work it is precisely
evaluated the corresponding to \eqref{eq:1p7} Bellman function for the case $q>1$.  Also in \cite{33} using the Monge-Amp\`{e}re equation approach a more general Bellman function than the one related to the dyadic Carleson imbedding Theorem has been precisely evaluated thus generalizing the corresponding result in \cite{10}. For more recent developments and results related to the Bellman function technique we refer to \cite{1}, \cite{6}, \cite{7}, \cite{22}, \cite{23}, \cite{24}, \cite{28}, \cite{29}, \cite{36}. Additional results can be found in \cite{2}, \cite{21}, \cite{34}, \cite{35}, while for the study of the general theory of maximal operators one can consult \cite{30}.

In this paper, as in our previous ones, we use Bellman functions as a means to gain deeper understanding of the corresponding maximal operators and we are
not using the standard techniques as Bellman dynamics and induction, corresponding PDE's, obstacle conditions etc. Instead, our methods being different from the Bellman function technique, we rely on the combinatorial structure of these operators. For such approaches, which enable us to study and solve problems such as the one which is described in this article one can see \cite{8}, \cite{9}, \cite{10}, \cite{11}, \cite{16} and \cite{19}. \bigskip

\section{Preliminaries} \label{sec:2}
Let $(X,\mu)$ be a nonatomic probability space. We give the following from \cite{10} or \cite{11}.
\begin{definition} \label{def:2p1}
A set $\mc T$ of measurable subsets of $X$ will be called a tree if the following are satisfied
\begin{enumerate}[i)]
\item $X\in\mc T$ and for every $I\in\mc T$, $\mu(I) > 0$.
\item For every $I\in\mc T$ there corresponds a finite of countable subset $C(I)$ of $\mc T$ containing at least two elements such that
\vspace{-5pt}
\begin{enumerate}[a)]
\item the elements of $C(I)$ are pairwise disjoint subsets of $I$.
\item $I = \cup\, C(I)$.
\end{enumerate}
\item $\mc T = \cup_{m\geq 0} \mc T_{(m)}$, where $\mc T_{(0)} = \left\{ X \right\}$ and $\mc T_{(m+1)} = \cup_{I\in \mc T_{(m)}} C(I)$.
\item The following holds
\[
\lim_{m\to\infty} \sup_{I\in \mc T_{(m)}} \mu(I) = 0
\]
\end{enumerate}
\end{definition}

\noindent We state now the following lemma as is given in \cite{10}.
\begin{lemma} \label{lem:2p1}
For every $I\in \mc T$ and every $\alpha\in (0,1)$ there exists a subfamily $\mc F(I) \subseteq \mc T$ consisting of pairwise disjoint subsets of $I$ such that
\[
\mu\!\left( \underset{J\in\mc F(I)}{\bigcup} J \right) = \sum_{J\in\mc F(I)} \mu(J) = (1-\alpha)\mu(I).
\]
\end{lemma}

\noindent Suppose now that we are given a tree $\mc T$ on a nonatomic probability space $(X,\mu)$. Then we define the associated dyadic maximal operator $\mc M_{\mc T}$ by \eqref{eq:1p5}. (see the Introduction).

\begin{definition} \label{def:2p2}
Let $(\phi_n)_n$ be a sequence of $\mu$-measurable nonnegative functions defined on $X,\ q\in(0,1)$, $0<h\leq f^q$ and $L\geq f$. Then $(\phi_n)_n$ is called extremal if the following hold $\int_X \phi_n\,\mr d\mu = f$, $\int_X \phi_n^q\,\mr d\mu = h$ for every $n\in \mb N$, and $\lim_n \int_X \left[ \max\left( \mc M_{\mc T}\phi_n,L\right) \right]^q\mr d\mu = c h$, where $c = \omega_q\!\left( \frac{(1-q)L^q + qL^{q-1}f}{h} \right)$. (See Theorem \ref{thm:1}, relation \eqref{eq:1p8}).
\end{definition}

For the proof of Theorem \ref{thm:1} an effective linearization was introduced for the operator $\mc M_{\mc T}$ valid for certain functions $\phi$.
We describe it. For $\phi \in L^1(X,\mu)$ nonnegative function and $I\in \mc T$ we define $\Av_I(\phi) = \frac{1}{\mu(I)} \int_I \phi\,\mr d\mu$. We will say that $\phi$ is $\mc T$-good if the set
\[
\mc A_\phi = \left\{ x\in X: \mc M_{\mc T}\phi(x) > \Av_I(\phi)\ \text{for all}\ I\in \mc T\ \text{such that}\ x\in I \right\}
\]
has $\mu$-measure zero.

\noindent Let now $\phi$ be $\mc T$-good and $x\in X\setminus\mc A_\phi$. We define $I_\phi(x)$ to be the largest in the nonempty set
\[
\big\{I\in \mc T: x\in T\ \text{and}\ \mc M_{\mc T}\phi(x) = \Av_I(\phi)\big\}.
\]
Now given $I\in\mc T$ let
\begin{align*}
A(\phi,I) &= \big\{ x\in X\setminus\mc A_\phi: I_\phi(x)=I \big\} \subseteq I,\ \text{and} \\[2pt]
S_\phi &= \big\{ I\in\mc T: \mu\left(A(\phi,I)\right) > 0 \big\} \cup \big\{X\big\}.
\end{align*}
Obviously then, $\mc M_{\mc T}\phi = \sum_{I\in S_\phi} \Av_I(\phi) \mc X_{A(\phi,I)}$, $\mu$-almost everywhere on $X$, where $\mc X_S$ is the characteristic function of $S\subseteq X$.
We define also the following correspondence $I \to I^\star$ by: $I^\star$ is the smallest element of $\{ J\in S_\phi: I\subsetneq J \}$.
This is defined for every $I\in S_\phi$ except $X$. It is obvious that the $A(\phi,I)$'s are pairwise disjoint and that $\mu\!\left( \cup_{I\notin S_\phi} A(\phi,I) \right)=0$, so that $\cup_{I\in S_\phi} A(\phi,I) \approx X$, where by $A\approx B$ we mean that $\mu(A\setminus B) = \mu(B\setminus A) = 0$. \\
We will need the following
\begin{lemma} \label{lem:2p2}
Let $\phi$ be $\mc T$-good and $I\in \mc T$, $I \neq X$. Then $I\in S_\phi$ if and only if every $J\in \mc T$ that contains properly $I$ satisfies $\Av_J(\phi) < \Av_I(\phi)$.
\end{lemma}

\begin{proof}
Suppose that $I\in S_\phi$. Then $\mu(A(\phi,I))>0$. As a consequence $A(\phi,I)\neq \emptyset$, so there exists $x\in A(\phi,I)$. By the definition of $A(\phi,I)$ we have that $I_\phi(x)=I$, that is $I$ is the largest element of $\mc T$ such that $\mc M_{\mc T}\phi(x) = \Av_I(\phi)$. As a consequence the implication stated in our lemma holds.
Conversely now, suppose that $I\in \mc T$ and for every $J\in \mc T$ with $J \supsetneq I$ we have that $\Av_J(\phi) < \Av_I(\phi)$. Then since $\phi$ is $\mc T$-good, for every $x\in I\setminus \mc A_\phi$ there exists $J_x$ ($= I_\phi(x)$) in $S_\phi$ such that $\mc M_{\mc T}\phi(x) = \Av_{J_x}(\phi)$ and $x\in J_x$. By our hypothesis we must have that $J_x \subseteq I$.
Now, consider the family $S' = \left\{ J_x,\ x\in I\setminus \mc A_\phi \right\}$. This has the property $\cup_{x\in I\setminus \mc A_\phi} J_x \approx I$. Choose a subfamily $S^2 = \{J_1, J_2, \ldots\}$ of $S'$, maximal under the $\subseteq$ relation. Then $I\approx \cup_{i=1}^\infty J_i$ where the last union
is pairwise disjoint because of the maximality of $S^2$.
Suppose now that $I$ does not belong to $S_\phi$. This means that $\mu(A(\phi,I))=0$, that is we must have that for every $x\in I\setminus\mc A_\phi$, $J_x\subsetneq I$. Since $J_x$ belongs to $S_\phi$ for every such $x$, by the first part of the proof of this Lemma we conclude that $\Av_{J_x}(\phi) > \Av_I(\phi)$.
Thus for every $i$, we must have that $\Av_{J_i}(\phi) > \Av_I(\phi)$. Since $S^2$ is a partition of $I$, we reach to a contradiction. Thus we must have that $I\in S_\phi$.
\end{proof}

Now the following is true, obtained in \cite{3}.
\begin{lemma} \label{lem:2p3}
Let $\phi$ be $\mc T$-good
\begin{enumerate}[i)]
\item If $I, J\in S_\phi$ then either $A(\phi,J)\cap I = \emptyset$ or $J\subseteq I$.
\item If $I\in S_\phi$, then there exists $J\in C(I)$ such that $J\notin S_\phi$.
\item For every $I\in S_\phi$ we have that
\[
I \approx \underset{\substack{J\in S_\phi \\ J\subseteq I\ \,}}{\cup} A(\phi,J).
\]
\item For every $I\in S_\phi$ we have that
\begin{gather*}
A(\phi,I) = I\setminus \underset{\substack{J\in S_\phi \\ J^\star = I\ }}{\cup} J,\ \ \text{so thet} \\
\mu(A(\phi,I)) = \mu(I) - \sum_{\substack{J\in S_\phi\\ J^\star=I\ }} \mu(J).
\end{gather*}
\end{enumerate}
\end{lemma}

\section{Some technical Lemmas}
In this section we collect some technical results whose proofs can be seen in \cite{4}. We begin with

\begin{lemma} \label{lem:3p1}
Let $0<q<1$ be fixed. Then
\begin{enumerate}[i)]
\item The function $\omega_q: [1,+\infty) \to [1,+\infty)$ is strictly increasing and strictly concave.
\item The function $U_q(x) = \frac{\omega_q(x)}{x}$ is strictly increasing on $[1,+\infty)$.
\end{enumerate}
\end{lemma}

\noindent For the next Lemma we consider for any $q\in (0,1)$ the following formula
\[
\sigma_q(k,x) = \frac{H_q\!\left(\frac{x(1-k)}{1-kx}\right)}{H_q(x)},
\]
defined for all $k, x$ such that $0<k<1$ and $0<x<\frac{1}{k}$. A straightforward computation shows (as is mentioned in \cite{4}) that
\[
\sigma_q(k,x) = \frac{(1-q)x + q - kx}{(1-k)^{1-q}(1-kx)^q((1-q)x+q)}.
\]
We state now the following
\begin{lemma} \label{lem:3p2}
\begin{enumerate}[i)]
\item For any fixed $\lambda>1$ the equation
\begin{equation} \label{eq:3p1}
H_q\!\left(\frac{x(1-k)}{1-kx}\right) = \lambda H_q(x),
\end{equation}
has a unique solution $x = \mc X(\lambda,k) = \mc X_\lambda(k)$ on the interval $\left(1,\frac 1 k\right)$ and it has a solution in the interval $(0,1)$ if and only if $\lambda<(1-k)^{q-1}$ in which case this is also unique.
\item For $\mu\geq 0$ define the following function
\begin{equation} \label{eq:3p2}
R_{q,\mu}(k,x) = \left(\frac{x(1-k)}{1-kx}\right)^q \frac{1}{\sigma_q(k,x)} + \left(\mu^q-x^q\right)(1-k),
\end{equation}
on $W=\big\{(k,x): 0<k<1\ \text{and}\ 1<x<\frac{1}{k}\big\}$. \\
Then if $\mu > 1$ and $\xi$ is in $(0,1]$ the maximum value of $R_{q,\mu}$ on the set $\left\{ (k,x)\in W: 0<k\leq \xi\ \text{and}\ \sigma_q(k,x)=\lambda\right\}$ is equal to $\frac{1}{\lambda}\omega_q(\lambda H_q(\mu))$ if $\xi\geq k_0(\lambda,\mu)$, where $k_0(\lambda,\mu)$ is given by
\begin{equation} \label{eq:3p3}
k_0(\lambda,\mu) = \frac{\omega_q(\lambda H_q(\mu))^\frac{1}{q} - \mu}{\mu\!\left(\omega_q(\lambda H_q(\mu))^\frac{1}{q}-1\right)},
\end{equation}
and is the unique in $\left(0, \frac 1 \mu\right)$ solution of the equation $\sigma_q(k_0,\mu)=\lambda$.
Additionally comparing with \rnum{1}) of this Lemma we have that $\mc X_\lambda(k_0)=\mu$.
\end{enumerate}
\end{lemma}

\noindent Now for the next Lemma we fix real numbers $f,\ h$ and $k$ with $0<h<f^q$ and $0<k<1$, and we consider the functions
\[
\ell_k(B) = (1-k)^{1-q}(f-B)^q + k^{1-q}B^q,
\]
defined for $0\leq B\leq f$ and
\begin{equation} \label{eq:3p4}
R_k(B) = \left\{ \begin{aligned}
& \left( h - (1-k)^{1-q}(f-B)^q \right) \omega_q\!\left( \frac{k^{1-q}B^q}{h - (1-k)^{1-q}(f-B)^q} \right), \\
& \hphantom{\frac{k^{1-q}B^q}{1-q},}\hspace{70pt} \text{if}\ (1-k)^{1-q}(f-B)^q< h\leq \ell_k(B), \\[5pt]
& \frac{k^{1-q}B^q}{1-q},\hspace{70pt} \text{if}\ h\leq (1-k)^{1-q}(f-B)^q,
\end{aligned} \right.
\end{equation}
defined for all $B\in [0,f]$ such that $\ell_k(B)\geq h$. \\
Noting that $\ell_k$ has an absolute maximum at $B=kf$ with $\ell_k(kf) = f^q > h$ and that it is monotone on each of the intervals $(0, kf)$ and $(kf, f)$ we conclude that either $\ell_k(f) \le h$ i.e. $k^{1-q}f^q < h$, in which case the equation $\ell_k(B) = h$ has a unique solution in $(kf,f)$ and this is denoted by $\rho_1 = \rho_1(f,h,k)$, or $\ell_k(f) \geq h$, in which case we set $\rho_1 = \rho_1(f,h,k) = f$.
Also either $\ell_k(0)<h$, i.e. $(1-k)^{1-q}f^q < h$ in which case the equation $\ell_k(B)=h$ has a unique solution on $(0,kf)$ and this is denoted by $\rho_0 = \rho_0(f,h,k)$, or $\ell_k(0)\geq h$ in which case we set $\rho_0 = \rho_0(f,h,k) = 0$. In all cases the domain of definition of $R_k$ is the interval $W_k = W_k(f,h) = [\rho_0,\rho_1]$. \\
We are now able to give the following
\begin{lemma} \label{lem:3p3}
The maximum value of the function $R_k$ on $W_k$ is attained at the unique point $B^\star = \mc X_\lambda(k)kf > kf$ where $\lambda = \frac{f^q}{h}$ (see Lemma \ref{lem:3p2}). Moreover
\begin{equation} \label{eq:3p5}
\max_{W_k}(R_k) = h\,\omega_q\!\left(\frac{f^q}{h} H_q(\mc X_\lambda(k)) \right) - (1-k)f^q(\mc X_\lambda(k))^q.
\end{equation}
Additionally $B^\star$ satisfies
\[
(1-k)^{1-q}(f-B^\star)^q < h < \ell_k(B^\star).
\]
\end{lemma}

The above Lemmas are enough for us to study the extremal sequences for \eqref{eq:1p8} as we shall see in the next Section.

\section{Extremal sequences for the Bellman function}
We prove the following
\begin{theorem} \label{thm:4p1}
Let $\phi$ be $\mc T$-good function such that $\int_X \phi\,\mr d\mu = f$. Let also $B=\{I_j\}_j$ be a family of pairwise disjoint elements of $S_\phi$, which is maximal on $S_\phi$ under $\subseteq$ relation. That is $I\in S_\phi \Rightarrow I\cap (\cup I_j) \neq \emptyset$. Then the following inequality holds
\begin{multline*}
\int_{X\setminus \cup_j I_j} (\mc M_{\mc T}\phi)^q\,\mr d\mu \leq \\
\frac{1}{(1-q)\beta} \left[ (\beta+1) \left(f^q - \sum\mu(I_j)y_{I_j}^q\right) - (\beta+1)^q \int_{X\setminus \cup_j I_j} \phi^q\,\mr d\mu \right]
\end{multline*}
for every $\beta>0$, where $y_{I_j} = \mathrm{Av}_{I_j}(\phi)$.
\end{theorem}

\begin{proof}
We follow \cite{4}.

Let $S = S_\phi$, $\alpha_I = \mu(A(\phi,I))$, $\rho_1 = \frac{\alpha_I}{\mu(I)}\in(0,1]$ and
\[
y_I = \Av_I(\phi) = \frac{1}{\mu(I)} \sum_{J\in S: J\subseteq I} \alpha_J x_J,\ \ \text{for every}\ \ I\in S,
\]
where $x_J=  \frac{1}{\alpha_J} \int_{A(\phi,J)} \phi\,d\mu$, for any $J \in S_\phi$.
It is easy now to see in view of Lemma \ref{lem:2p3} \rnum{4}) that
\[
y_I\mu(I) = \sum_{J\in S: J^\star=I} y_J\mu(J) + \alpha_I x_I,
\]
and so by using concavity of the function $t\to t^q$, we have for any $I\in S$,
\begin{align} \label{eq:4p1}
[y_i\mu(I)]^q &= \left( \sum_{J\in S: J^\star=I} y_J\mu(J) + \alpha_Ix_I\right)^q \notag \\
 &= \left( \sum_{J\in S: J^\star = I} \tau_I\mu(J)\frac{y_I}{\tau_I} + \sigma_I\alpha_I\frac{x_I}{\sigma_I}\right)^q \notag \\
 &\geq \sum_{J\in S: J^\star = I} \tau_I\mu(J)\left(\frac{y_J}{\tau_I}\right)^q + \sigma_I\alpha_I\left(\frac{x_I}{\sigma_I}\right)^q,
\end{align}
where $\tau_I, \sigma_I > 0$ satisfy
\[
\tau_I(\mu(I)-\alpha_I) + \sigma_I\alpha_I = \sum_{J\in S: J^\star=I} \tau_I\mu(J) + \sigma_I\alpha_I = 1.
\]
We now fix $\beta>0$ and let
\[
\sigma_I = ((\beta+1)\mu(I) - \beta\alpha_I)^{-1},\ \ \tau_I = (\beta+1)\sigma_I
\]
which satisfy the above relation and thus we get by dividing with $\sigma_I^{1-q}$ that
\begin{equation} \label{eq:4p2}
((\beta+1)\mu(I) - \beta\alpha_I)^{1-q} (y_I\mu(I))^q \geq \sum_{J\in S: J^\star=I} (\beta+1)^{1-q} \mu(J) y_J^q + \alpha_Ix_I^q,
\end{equation}
However,
\begin{equation} \label{eq:4p3}
x_I^q = \left( \frac{1}{\alpha_I} \int_{A(\phi,I)} \phi\,d\mu\right)^q \geq \frac{1}{\alpha_I}\int_{A(\phi,I)} \phi^q\,\mr d\mu.
\end{equation}
We sum now \eqref{eq:4p2} over all $I\in S$ such that $I\supsetneq I_j$ for some $j$ (which we denote by $I\supsetneq \mr{piece}(B)$) and we obtain
\begin{multline} \label{eq:4p4}
\sum_{I\supsetneq \mr{piece}(B)} ((\beta+1)\mu(I) - \beta\alpha_I)^{1-q} (y_I\mu(I))^q \geq\\
\sum_{\substack{I\supsetneq \mr{piece}(B)\\ I\neq X}} (\beta+1)^{1-q} \mu(I) y_I^q + \sum_j (\beta+1)^{1-q} \mu(I_j) y_{I_j}^q + \sum_{I\supsetneq \mr{piece}(B)} a_I x_I^q.
\end{multline}
Note that the first two sums are produced in \eqref{eq:4p4} because of maximality of $(I_j)$. \eqref{eq:4p4} now gives:
\begin{multline} \label{eq:4p5}
\sum_{I\supsetneq \mr{piece}(B)} (\beta+1)^{1+q}\mu(I)y_I^q - \sum_{I\supsetneq \mr{piece}(B)} ((\beta+1)\mu(I) - \beta\alpha_I)^{1-q} (y_I\mu(I))^q \leq \\
(\beta+1)^{1-q}y_X^q - \int_{X\setminus \cup I_j} \phi^q\,\mr d\mu - \sum_j (\beta+1)^{1-q}\mu(I_j)y_{I_j}^q,
\end{multline}
in view of H\"{o}lder's inequality \eqref{eq:4p3}. Thus \eqref{eq:4p5} gives
\begin{multline} \label{eq:4p6}
\sum_{I\supsetneq \mr{piece}(B)} \left[(\beta+1)^{1-q}\mu(I) - ((\beta+1)\mu(I) - \beta\alpha_I)^{1-q}\mu(I)^q\right]y_I^q \leq \\
(\beta+1)^{1-q} \left(f^q - \sum\mu(I_j)y_{I_j}^q\right) - \int_{X\setminus \cup I_j} \phi^q\,\mr d\mu,
\end{multline}
On the other side we have that
\begin{align} \label{eq:4p7}
& \frac{1}{\mu(I)} \left[ (\beta+1)^{1-q}\mu(I) - ((\beta+1)\mu(I) - \beta\alpha_I)^{1-q}\mu(I)^q\right]  =  \notag \\
& (\beta+1)^{1-q} - ((\beta+1) - \beta\rho_I)^{1-q} \geq
 (1-q)(\beta+1)^{-q}\beta\rho_I = \notag \\
& (1-q)(\beta+1)^{-q}\beta\frac{\alpha_I}{\mu(I)},
\end{align}
where the inequality in \eqref{eq:4p7} comes from the differentiation mean value theorem on calculus.

From the last two inequalities we conclude
\begin{equation} \label{eq:4p8}
(1-q)(\beta+1)^{-q}\beta \!\! \sum_{I\supsetneq \mr{piece}(B)} a_Iy_I^q \leq
(\beta+1)^{1-q} \left( f^q - \sum\mu(I_j)y_{I_j}^q\right) - \int_{X\setminus \cup I_j} \!\!\phi^q\,\mr d\mu.
\end{equation}
Now it is easy to see that
\[
\sum_{I\supsetneq \mr{piece}(B)} a_Iy_I^q = \int_{X\setminus \cup I_j} (\mc M_{\mc T}\phi)^q\,\mr d\mu,
\]
because $B=\{I_j\}_j$ is a family of elements of $S_\phi$. Then \eqref{eq:4p8} becomes
\begin{multline*}
\int_{X\setminus \cup I_j} (\mc M_{\mc T}\phi)^q\,\mr d\mu \leq \\
\frac{1}{(1-q)\beta} \left[ (\beta+1)\bigg(f^q - \sum_j \mu(I_j)y_{I_j}^q\bigg) - (\beta+1)^q\int_{X\setminus \cup I_j} \phi^q\,\mr d\mu \right]
\end{multline*}
for any fixed $\beta>0$, and $\phi: \mc T$-good.

In this way we derived the proof of Theorem \ref{thm:4p1}.
\end{proof}

In the same lines as above we can prove:
\begin{theorem} \label{thm:4p2}
Let $\phi$ be $\mc T$-good and $\mc A=\{I_j\}$ be a pairwise disjoint family of elements of $S_\phi$. Then for every $\beta>0$ we have that:
\[
\int_{\cup I_j} (\mc M_{\mc T}\phi)^q\,\mr d\mu \leq
\frac{1}{(1-q)\beta} \left[ (\beta+1) \sum \mu(I_j)y_{I_j}^q - (\beta+1)^q \int_{\cup I_j} \phi^q\,\mr d\mu \right].
\]
\end{theorem}

\begin{proof}
We use the technique mentioned above in Theorem \ref{eq:4p1} by summing inequality \eqref{eq:4p2} with respect to all $I\in S_\phi$ with $I\subseteq I_j$ for any $j$. The rest details are easy to be verified.
\end{proof}

We have now the following generalization of Theorem \ref{thm:4p1}.
\begin{corollary} \label{cor:4p1}
Let $\phi$ be $\mc T$-good and $\mc A = \{I_j\}$ be a pairwise disjoint family of elements of $S_\phi$. Then for every $\beta>0$
\begin{multline} \label{eq:4p9}
\int_{X\setminus\cup I_j} (\mc M_{\mc T}\phi)^q\,\mr d\mu \leq
\frac{1}{(1-q)\beta} \left[ (\beta+1) \left( f^q - \sum\mu(I_j) y_{I_j}^q\right) - \right. \\
\left. (\beta+1)^q \int_{X\setminus\cup I_j} \phi^q\,\mr d\mu \right],
\end{multline}
where $y_{I_j} = Av_{I_j}(\phi)$, $f=\int_X \phi\,\mr d\mu$.
\end{corollary}

\begin{proof}
We choose a pairwise disjoint family $(J_i)_i = B \subseteq S_\phi$ such that the union $\mc A\cup B$ is maximal under the relation $\subseteq$ in $S_\phi$, and $I_j\cap J_i = \emptyset$ for all $i, j$.
Then if we apply Theorem \ref{thm:4p1} for $\mc A\cup B$ and Theorem \ref{thm:4p2} for $B$, and sum the two inequalities we derive the proof of our Corollary.
\end{proof}

\noindent We now proceed to the
\begin{proof}[Proof of Theorem \ref{thm:a}]
Suppose that we are given an extremal sequence $\phi_n: (X,\mu)\to \mb R^+$ of functions, such that $\int_X \phi_n\,\mr d\mu = f$, $\int_X \phi_n^q\,\mr d\mu = h$ for any $n\in \mb N$ and
\begin{equation} \label{eq:4p10}
\lim_n \int_X \left[ \max(\mc M_{\mc T}\phi_n, L)\right]^q\mr d\mu = h c.
\end{equation}
We prove that
\begin{equation} \label{eq:4p11}
\lim_n \int_X \left| \max(\mc M_{\mc T}\phi_n,L) - c^{\frac{1}{q}}\phi_n \right|^q\mr d\mu = 0.
\end{equation}
For the proof of \eqref{eq:4p11} we are going to give the chain of inequalities from which one gets Theorem \ref{thm:1}.
Then we use the fact that these inequalities become equalities in the limit. \\
Fix a $n\in\mb N$ and write $\phi=\phi_n$. For this $\phi$ we have the following
\begin{equation} \label{eq:4p12}
I_\phi := \int_X \left[ \max(\mc M_{\mc T}\phi,L) \right]^q\mr d\mu =
\int_{\{\mc M_{\mc T}\phi \geq L\}} (\mc M_{\mc T}\phi)^q\,\mr d\mu + L^q(1-\mu(E_\phi))
\end{equation}
where $E_\phi = \{\mc M_{\mc T}\phi \geq L\}$. \\
We write $E_\phi$ as $E_\phi = \cup I_j$, where $I_j$ are maximal elements of the $\mc T$, such that
\begin{equation} \label{eq:4p13}
\frac{1}{\mu(I_j)} \int_{I_j} \phi\,\mr d\mu \geq L.
\end{equation}
We set for any $j$, $\alpha_j = \int_{I_j} \phi^q\,\mr d\mu$ and $\beta_j = \mu(I_j)^{1-q} \left( \int_{I_j} \phi\,\mr d\mu\right)^q$. Additionally we set $A = \sum \alpha_j = \int_E \phi^q\,\mr d\mu \leq h$, where $E := E_\phi$, and $B = \sum_j \left( \mu(I_j)^{q-1} \beta_j\right)^\frac{1}{q} = \int_E \phi\,\mr d\mu \leq f$. \\
We also set $k=\mu(E)$. Note that the variables $A$, $B$, $k$ depend on the function $\phi$. \\
From \eqref{eq:4p12} we now obtain
\begin{equation} \label{eq:4p14}
I_\phi = L^q(1-k) + \sum_j \int_{I_j} (\mc M_{\mc T}\phi)^q\,\mr d\mu.
\end{equation}
Note now that from the maximality of any $I_j$ we have that $\mc M_{\mc T}\phi(x) =\\ \mc M_{\mc T(I_j)}\phi(x)$, for every $x\in I_j$ where $\mc T(I_j) = \{ J\in\mc T: J\subseteq I_j \}$.
We now apply Theorem \ref{thm:a} for the measure space $\left( I_j, \frac{\mu(\cdot)}{\mu(I_j)}\right)$ and for $L=\frac{1}{\mu(I_j)} \int_{I_j}\phi\,\mr d\mu = \Av_{I_j}(\phi)$, for any $j$, and we get that
\begin{equation}
I_\phi \leq L^q(1-k) + \sum_j \alpha_j\, \omega_q\! \left(\frac{\beta_j}{\alpha_j}\right).
\end{equation}
Note that $k^{1-q} B^q = \left(\sum_j \mu(I_j)\right)^{1-q} \left(\sum_j \left(\mu(I_j)^{q-1} \beta_j\right)^\frac{1}{q}\right)^q \geq \sum \beta_j \geq A$ in view of H\"{o}lder's inequality. \\
We now use the concavity of the function $\omega_q$, as can be seen in Lemma \ref{lem:3p1} \rnum{1}), and we conclude that
\begin{equation} \label{eq:4p16}
I_\phi \leq L^q(1-k) + A\,\omega_q\!\left(\frac{\sum \beta_j}{A}\right) \leq L^q(1-k) + A\,\omega_q\!\left(\frac{k^{1-q}B^q}{A}\right),
\end{equation}
where the last inequality comes from the fact that $\omega_q$ is increasing.
It is not difficult to see that the parameters $A, B$ and $k$ satisfy the following inequalities:
\begin{align*}
& A \leq k^{1-q} B^q,\quad A\leq h,\quad B\leq f,\quad 0\leq k\leq 1\quad \text{and} \\
& h-A \leq (1-k)^{1-q}(f-B)^q,
\end{align*}
the last one being  $\int_{X\setminus E} \phi^q\,\mr d\mu \leq \mu(X\setminus E)^{1-q} \left(\int_{X\setminus E} \phi\,\mr d\mu\right)^q$.
It is also easy to see that $B\geq kL$, by \eqref{eq:4p13} and the disjointness of $\{I_j\}_j$. From the above inequalities and \eqref{eq:4p16} we conclude that
\begin{equation} \label{eq:4p17}
I_\phi \leq L^q(1-k) + R_k(B),
\end{equation}
where $R_k$ is given by \eqref{eq:3p4}. Thus using Lemma \ref{lem:3p3} we have that
\begin{multline} \label{eq:4p18}
I_\phi \leq L^q(1-k) + R_k(B^\star) = \\
L^q(1-k) + h\,\omega_q\!\left(\frac{f^q}{h} H_q(\mc X_\lambda(k))\right) - (1-k)f^q\left(\mc X_\lambda(k)\right)^q,
\end{multline}
where $\lambda = \frac{f^q}{h}$, $\mc X_\lambda(k)$ is given in Lemma \ref{lem:3p2} and $B^\star = \mc X_\lambda(k) k f > k f$. According to Lemma \ref{lem:3p2} $\mc X_\lambda(k) $ satisfies $1< \mc X_\lambda(k)< \frac{1}{k}$ and
\[
H_q\!\left(\frac{\mc X_\lambda(k)(1-k)}{1-k\mc X_\lambda(k)}\right) = \lambda H_q(x).
\]
From \eqref{eq:4p18} we have that
\begin{equation} \label{eq:4p19}
I_\phi \leq \left[ L^q - f^q(\mc X_\lambda(k))^q\right](1-k) + h\,\omega_q\!\left(\frac{f^q}{h} H_q(\mc X_\lambda(k))\right).
\end{equation}
We now set $\mu = \frac{L}{f} > 1$. Then \eqref{eq:4p19} becomes
\begin{equation} \label{eq:4p20}
I_\phi \leq f^q \left\{ \left[ \mu^q - (\mc X_\lambda(k))^q\right](1-k) + \omega_q\!\left(\frac{f^q}{h} H_q(\mc X_\lambda(k))\right)\frac{1}{\sigma_q(k,\mc X_\lambda(k))}\right \}
\end{equation}
Remember that $\mc X_\lambda(k)$ satisfies $\sigma_q(k,\mc X_\lambda(k)) = \lambda = \frac{f^q}{h}$ by Lemma \ref{lem:3p2}. Now by the last equation we have that
\begin{multline} \label{eq:4p21}
\frac{f^q}{h} H_q(\mc X_\lambda(k)) =
H_q\!\left(\frac{\mc X_\lambda(k)(1-k)}{1-k\mc X_\lambda(k)}\right) \implies \\
\omega_q\!\left(\frac{f^q}{h}H_q(\mc X_\lambda(k))\right) =
\omega_q\!\left(H_q\!\left(\frac{\mc X_\lambda (k)(1-k)}{1-k\mc X_\lambda(k)}\right)\right).
\end{multline}
Remember that $\omega_q(z) = \left(H_q^{-1}(z)\right)^q$, for any $z\geq 1$. Thus $\omega_q\!\left(\frac{f^q}{h}H_q(\mc X_\lambda(k))\right) = \left(\frac{\mc X_\lambda(k)(1-k)}{1-k\mc X_\lambda(k)}\right)^q$. Thus from \eqref{eq:4p20} we have as a consequence that
\begin{align} \label{eq:4p22}
I_\phi &= f^q \left\{ \left[ \mu^q - \mc X_\lambda(k)^q\right](1-k) + \frac{1}{\sigma_q(k,\mc X_\lambda(k))} \left(\frac{(1-k)\mc X_\lambda(k)}{1-k\mc X_\lambda(k)}\right)^q\right\} \notag \\
 &= f^q R_{q,\mu}(k,\mc X_\lambda(k)).
\end{align}
According then to Lemma \ref{lem:3p2} \rnum{2}) we have that
\begin{multline} \label{eq:4p23}
I_\phi \leq f^q\left\{ \frac{1}{\lambda} \omega_q(\lambda H_q(\mu))\right\} =
h\,\omega_q\!\left(\frac{f^q}{h} H_q\!\left(\frac{L}{f}\right)\right) = \\
h\,\omega_q\!\left(\frac{(1-q)L^q + qL^{q-1}f}{h}\right) = h c = B_q^{\mc T}(f,h,L,1).
\end{multline}
Now if $\phi$ runs along $(\phi_n)$, we see by the extremality of this sequence that in the limit we have equality in \eqref{eq:4p23}.
That is we have equalities in the limit to all the previous steps which lead to \eqref{eq:4p23}. \\
If we let now $\phi=\phi_n$, we write $A=A_n,\ B=B_n$ and $k=k_n$. Since we have equality in the last inequality giving \eqref{eq:4p23} we conclude that $k\to k_0$, where $k_0$ satisfies:
\[
k_0(\lambda,\mu) = \frac{\omega_q(\lambda H_q(\mu))^{\frac{1}{q}} - \mu}{\mu\left(\omega_q(\lambda H_q(\mu))^\frac{1}{q} - 1\right)} \quad \text{and} \quad
\mc X_\lambda(k_0) = \mu = \frac{L}{f}.
\]
Additionally we must have that $B_n \to B^\star = k_0 f \mc X_\lambda(k_0) = k_0 f \frac{L}{f} = k_0 L$, which means exactly that $\lim \frac{1}{\mu(E_n)} \int_{E_n} \phi\,\mr d\mu = L$, where $E_n = \{\mc M_{\mc T}\phi_n \geq L\}$, with $\mu(E_n) = k_n \to k_0$.
This gives us equality in the weak type inequality for $(\phi_n)_n$, in case where $\lambda=L$. \\
We wish to prove that if we define $I_n^{(1)} := \int_X \big| \max(\mc M_{\mc T}\phi_n, L) - c^{\frac{1}{q}}\phi_n\big|^q\mr d\mu$, we then have that $\lim_n I_n^{(1)} = 0$.
Thus we write
\begin{align*}
I_n^{(1)} &= \int_{E_n} \left| \mc M_{\mc T}\phi_n - c^\frac{1}{q}\phi_n\right|^q\mr d\mu + \int_{X\setminus E_n} \left|L-c^\frac{1}{q}\phi_n\right|^q\mr d\mu \\
 &= J_n + \Lambda_n,
\end{align*}
where $J_n$ and $\Lambda_n$ have the obvious meaning.
Remember now in the above chain of inequalities leading to \eqref{eq:4p23}, we have already proved that
\[
\int_X [\max(\mc M_{\mc T}\phi_n,L)]^q\,\mr d\mu < L^q(1-k_n) + A_n \omega_q\!\left(\frac{k_n^{1-q}B_n^q}{A_n}\right)
\]
and  that we used the fact that $A_n\omega_q\!\left(\frac{k_n^{1-q}B_n^q}{A_n}\right) \leq R_k(B_n)$. Thus we must have in the inequality $A_n \geq h-(1-k)^{1-q}(f-B_n)^q=C_n$, equality in the limit according to the way that $R_k(B)$ is defined.
Thus we must have that $h-A_n \approx (1-k_n)^{1-q}(f-B_n)^q$, or equivalently
\begin{equation} \label{eq:4p24}
\bigg( \frac{1}{\mu(X\!\setminus\! E_n)} \int_{X\setminus E_n} \phi\,\mr d\mu\bigg)^q \approx \frac{1}{\mu(X\!\setminus\! E_n)} \int_{X\setminus E_n} \phi_n^q\,\mr d\mu.
\end{equation}
Additionally we must have that
\begin{equation} \label{eq:4p25}
\int_{E_n} (\mc M_{\mc T}\phi_n)^q\,\mr d\mu \approx A_n\,\omega_q\!\left(\frac{k_n^{1-q}B_n^q}{A_n}\right).
\end{equation}
We first prove that $\Lambda_n = \int_{X\setminus E_n} \big|L-c^\frac{1}{q}\phi_n\big|^q\,\mr d\mu \to 0$, as $n\to \infty$.
Since $\int_{E_n}\phi_n\,\mr d\mu = B_n \to B^\star = L k_0$, we must have that $\frac{1}{\mu(X\setminus E_n)} \int_{X\setminus E_n} \phi_n\,\mr d\mu = \frac{f-B_n}{1-k_n} \to \frac{f-B^\star}{1-k_0} = \frac{f-L k_0}{1-k_0}$. \\
By the properties that $k_0$ satisfies, we have that
\begin{equation} \label{eq:4p26}
k_0 = k_0(\lambda,\mu) = \frac{\omega_q(\lambda H_q(\mu))^\frac{1}{q} - \mu}{\mu\left(\omega_q(\lambda H_q(\mu))^\frac{1}{q} - 1\right)},
\end{equation}
where $\lambda = \frac{f^q}{\mu}$, $\mu = \frac{L}{f}$. Of course $\omega_q\!\left( \frac{f^q}{h} H_q\!\left(\frac{L}{f}\right)\right) = c$, thus \eqref{eq:4p26} gives
\[
k_0 = \frac{c^\frac{1}{q} - \frac{L}{f}}{\frac{L}{f}(c^\frac{1}{q}-1)} = \frac{f c^\frac{1}{q} - L}{L c^\frac{1}{q} - L} \implies \frac{f-k_0L}{1-k_0} = \frac{L}{c^\frac{1}{q}}.
\]
Thus \eqref{eq:4p24} becomes:
\begin{equation} \label{eq:4p27}
\left[ \frac{1}{\mu(X\!\setminus\! E_n)} \int_{X\setminus E_n} \phi_n^q\,\mr d\mu\right]^\frac{1}{q} \approx
\left( \frac{1}{\mu(X\!\setminus\! E_n)} \int_{X\setminus E_n} \phi_n\,\mr d\mu\right) \cong
\frac{L}{c^\frac{1}{q}} =: \tau
\end{equation}
In order to show that $\Lambda_n \to 0$, as $n\to \infty$ it is enough to prove that $\int_{X\setminus E_n} |\phi_n-\tau|\,\mr d\mu \to 0$, as $n\to\infty$, where $\tau$ is defined as above.
We use now the following elementary inequality
\begin{equation} \label{eq:4p28}
t + \frac{1-q}{q} \geq \frac{t^q}{q},
\end{equation}
which holds for every $q\in(0,1)$ and every $t>0$.
Additionally we have equality in \eqref{eq:4p28} only if $t=1$. We also assume that $\tau=1$ on \eqref{eq:4p27}. We can overcome this difficulty by dividing \eqref{eq:4p27} by $\tau$ and by considering $\frac{\phi_n}{\tau}$ instead of $\phi_n$.\\
By \eqref{eq:4p28} we have that
\begin{equation} \label{eq:4p29}
\frac{\phi_n^q(x)}{q} \leq \frac{1-q}{q} + \phi_n(x),\quad \text{for all}\ \ (X\!\setminus\! E_n)\cap \{\phi_n>1\}
\end{equation}
and that
\begin{equation} \label{eq:4p30}
\frac{\phi_n^q(y)}{q} \leq \frac{1-q}{q} + \phi_n(y),\quad \text{for all}\ \ y\in (X\!\setminus\! E_n) \cap \{\phi_n \leq 1\}.
\end{equation}
By integrating in the respective domains inequalities \eqref{eq:4p29} and \eqref{eq:4p30} we immediately get:
\begin{align} \label{eq:4p31}
\frac{1}{q} \int\limits_{(X\setminus E_n)\cap\{\phi_n > 1\}}\!\! \phi_n^q\,\mr d\mu \leq \frac{1-q}{q}\,\mu\!\left( (X\!\setminus\! E_n) \cap \{\phi_n>1\} \right) + \int\limits_{(X\setminus E_n)\cap\{\phi_n>1\}}\!\! \phi_n\,\mr d\mu, \\
\frac{1}{q} \int\limits_{(X\setminus E_n)\cap\{\phi_n\leq 1\}}\!\! \phi_n^q\,\mr d\mu \leq \frac{1-q}{q}\,\mu\!\left( (X\!\setminus\! E_n) \cap \{\phi_n \leq 1\} \right) + \int\limits_{(X\setminus E_n)\cap\{\phi_n \leq 1\}}\!\! \phi_n\,\mr d\mu. \label{eq:4p32}
\end{align}
Adding \eqref{eq:4p31} and \eqref{eq:4p32} we conclude that
\begin{equation} \label{eq:4p33}
\frac{1}{q}\frac{1}{\mu(X\!\setminus\! E_n)} \int_{X\setminus E_n} \phi_n^q\,\mr d\mu \leq \frac{1-q}{q} + \frac{1}{\mu(X\!\setminus\! E_n)} \int_{X\setminus E_n} \phi_n\,\mr d\mu.
\end{equation}
Since now \eqref{eq:4p27} holds, with $\tau=1$, we conclude that we have equality in \eqref{eq:4p33} in the limit.
Thus we must have equality in the limit in both of \eqref{eq:4p31} and \eqref{eq:4p32}. Thus we have that
\begin{align} \label{eq:4p34}
\frac{1}{\mu((X\!\setminus\! E_n)\cap\{\phi_n > 1\})} \int_{(X\setminus E_n)\cap\{\phi_n > 1\}} \phi_n\,\mr d\mu &\approx 1\ \ \text{and} \notag \\
\frac{1}{\mu((X\!\setminus\! E_n)\cap\{\phi_n \leq 1\}} \int_{(X\setminus E_n)\cap\{\phi_n \leq 1\}} \phi_n\,\mr d\mu &\approx 1.
\end{align}
Then from \eqref{eq:4p34} we have as a consequence that
\begin{multline*}
\int\limits_{(X\setminus E_n)\cap\{\phi_n > 1\}} (\phi_n-1)\,\mr d\mu =
\mu((X\!\setminus\! E_n)\cap\{\phi_n > 1\})\, \cdot \\
\Bigg\{\frac{1}{\mu((X\!\setminus\! E_n)\cap\{\phi_n > 1\})} \int\limits_{(X\setminus E_n)\cap\{\phi_n > 1\}}\!\! \phi_n\,\mr d\mu - 1\Bigg\}
\end{multline*}
tends to zero, as $n\to\infty$. \\
By the same way $\int_{(X\setminus E_n)\cap\{\phi_n \leq 1\}} (1-\phi_n)\,\mr d\mu \to 0$, so as a result we have $\int_{X\setminus E_n} |\phi_n-1|\,\mr d\mu \approx 0$. Since now  $\int_{X\setminus E_n} |\phi_n-1|^q\,\mr d\mu \leq \mu(X\setminus E_n)^{1-q} \big[\int_{X\setminus E_n} |\phi_n-1|\,\mr d\mu\big]^q$ and $\mu(E_n)\to k_0\in (0,1)$
we have that $\lim_n \int_{X\setminus E_n} |\phi_n-1|^q\,\mr d\mu = 0$. \\
By the above reasoning we conclude $\Lambda_n = \int_{X\setminus E_n} |L-c^\frac{1}{q}\phi_n|^q\,\mr d\mu \to 0$, as $n\to \infty$.
\end{proof}

\noindent We now prove the following
\begin{theorem} \label{thm:4p3}
Let $(\phi_n)_n$ be extended, where $0 < h \leq L^q$, $L \geq f$ are fixed. Consider for each $n\in\mb N$ a pairwise disjoint family $\mc A_n = (I_{j,n})_j$ such that the following limit exists:
\[
\lim_n \sum_{I\in \mc A_n} \mu(I) y_{I,n}^q,\ \ \text{where}\ \ y_{I,n} = \Av_I(\phi_n),\ I\in \mc A_n.
\]
Suppose also that $\cup\mc A_n = \cup_j I_{j,n} \subseteq \{\mc M_{\mc T}\phi_n \geq L\}$, for each $n=1, 2, \ldots$. Then $\lim_n \int_{\cup\mc A_n} (\mc M_{\mc T}\phi_n)^q\,\mr d\mu = c \lim_n \int_{\cup\mc A_n} \phi_n^q\,\mr d\mu$, where $c = \omega_q\!\left(\frac{(1-q)L^q + qL^{1-q}f}{h}\right)$.
\end{theorem}

\begin{proof}
Define $\ell_n = \sum_{I\in\mc A_n} \mu(I) y_{I,n}^q$. By Theorem \ref{eq:4p2} we immediately see that for each $n\in\mb N^\star$
\begin{equation} \label{eq:4p35}
\int_{\cup\mc A_n} (\mc M_{\mc T}\phi_n)^q\,\mr d\mu \leq
\frac{1}{(1-q)\beta} \left[ (\beta+1) \sum_{I\in\mc A_n} \mu(I)y_{I,n}^q - (\beta+1)^q \int_{\cup\mc A_n} \phi_n^q\,\mr d\mu \right].
\end{equation}
Suppose that $E_n = \{\mc M_{\mc T}\phi_n \geq L\} = \cup I_j^{(n)}\ n=1, 2, \ldots$ where $I_j^{(n)}\in S_{\phi_n}$, for each $j$. \\
Now by our hypothesis we have that $\cup\mc A_n\subseteq E_n$, for all $n\in\mb N^\star$. Thus
\[
E_n\setminus \cup\mc A_n = \cup_j\left[ I_j^{(n)}\setminus \cup\mc A_n\right].
\]
Consider now for each $j$ and $n$ the probability since $\left( I_j^{(n)}, \frac{\mu(\cdot)}{\mu(I_j^{(n)})}\right)$, and apply there Theorem \ref{thm:4p1}, to get after summing on $j$ the following inequality
\begin{multline} \label{eq:4p36}
\int_{E_n\setminus \cup\mc A_n}(\mc M_{\mc T}\phi_n)^q\,\mr d\mu \leq \\
\frac{1}{(1-q)\beta} \bigg\{ (\beta+1) \bigg[ \sum_{I\in \{I_j^{(n)}\}_j} \mu(I) y_{I,n}^q - \sum_{I\in\mc A_n} \mu(I)y_{I,n}^q\bigg] - (\beta+1)^q\!\!\! \int\limits_{E_n\setminus \cup\mc A_n}\! \phi_n^q\,\mr d\mu \bigg\}.
\end{multline}
Summing \eqref{eq:4p35} and \eqref{eq:4p36} we have as a consequence that:
\begin{equation} \label{eq:4p37}
\int_{E_n} (\mc M_{\mc T}\phi_n)^q\,\mr d\mu \leq
\frac{1}{(1-q)\beta} \bigg[ (\beta+1) \sum_{I\in \{I_j^{(n)}\}_j} \mu(I)y_{I,n}^q - (\beta+1)^q\int_{E_n}\phi_n^q\,\mr d\mu \bigg].
\end{equation}
Using now the concavity of $t\mapsto t^q$, for $q\in(0,1)$ we obtain the inequality
\begin{equation} \label{eq:4p38}
\sum_{I\in\{I_j^{(n)}\}_j} \mu(I)y_{I,n}^q \leq
\frac{\left(\sum_{I\in\{I_j^{(n)}\}_j} \mu(I)y_{I,n}\right)^q}{\left(\sum_{I\in\{I_j^{(n)}\}_j} \mu(I)\right)^{q-1}} =
\frac{\left(\int_{E_n} \phi_n\,\mr d\mu\right)^q}{\mu(E_n)^{q-1}}.
\end{equation}
Thus \eqref{eq:4p37} in view of \eqref{eq:4p38} gives
\begin{multline} \label{eq:4p39}
\int_{E_n} (\mc M_{\mc T}\phi_n)^q\,\mr d\mu \leq
\frac{1}{(1-q)\beta} \left[ (\beta+1) \frac{1}{\mu(E_n)^{q-1}} \left( \int_{E_n} \phi_n\,\mr d\mu \right)^q - \right. \\
\left. (\beta+1)^q\int_{E_n} \phi_n^q\,\mr d\mu \right]
\end{multline}
By our hypothesis we have that
\begin{equation} \label{eq:4p40}
\int_{E_n} (\mc M_{\mc T}\phi_n)^q\,\mr d\mu \approx \\
\int_{E_n}\phi_n^q\,\mr d\mu\cdot \omega_1\!\left(\frac{k_n^{1-q}\left(\int_{E_n}\phi_n\,\mr d\mu\right)^q}{\int_{E_n}\phi_n^q\,\mr d\mu}\right),
\end{equation}
since $(\phi_n)$ is extremal, where $k_n=\mu(E_n)$, for all $n\in\mb N$. \\
But then by the definition of $\omega_q$; this means exactly that we have equality in the limit in \eqref{eq:4p39} for $\beta=\beta_n=\omega_q\!\left(\frac{k_n^{1-q}\left(\int_{E_n}\phi_n\,\mr d\mu\right)^q}{\int_{E_n}\phi_n^q\,\mr d\mu}\right)^\frac{1}{k}-1$ (see (3.18) and (3.19) in \cite{4}). \\
We set $c_{1,n} = \frac{k_n^{1-q}\left(\int_{E_n}\phi_n\,\mr d\mu\right)^q}{\int_{E_n}\phi_n^q\,\mr d\mu}$. We now prove that $c_{1,n}\to \frac{(1-q)L^q + qL^{q-1}f}{h}$, as $n\to \infty$.
Indeed, note that by the chain of inequalities leading to the least upper bound $B_q^{\mc T}(f,h,L,1) = c\,h$, we must have that
\begin{equation} \label{eq:4p41}
L^q(1-k_0) + \omega_q(c_{1,n})\int_{E_n}\phi_n^q\,\mr d\mu \approx c\,h,\quad \text{as}\ \ n\to\infty,
\end{equation}
where we suppose that $k_n\to k_0$ (we pass to a subsequence if necessary).
Now \eqref{eq:4p41} can be written as
\begin{equation} \label{eq:4p42}
L^q(1-k_n) + \omega_q(c_{1,n})\int_{E_n}\phi_n^q\,\mr d\mu \approx
\left(h - \int_{E_n}\phi_n^q\,\mr d\mu\right)c + \int_{E_n}\phi_n^q\,\mr d\mu\!\cdot\! c.
\end{equation}
But as we have already proved before, we have
\begin{align*}
\begin{aligned}
& L^q(1-k_n) \approx
 \left(h-\int_{E_n}\phi_n^q\,\mr d\mu\right)c \iff
 \frac{1}{1-k_n} \left(h - \int_{E_n}\phi_n^q\,\mr d\mu\right)= \frac{L^q}{c} \iff \\
& \frac{1}{\mu(X\!\setminus\! E_n)} \int_{X\setminus E_n}\phi_n^q\,\mr d\mu = \frac{L^q}{c}\ \text{which is indeed right, since}
 \end{aligned} &\\
 \int_{X\setminus E_n} \left|\phi_n - \frac{L}{c^\frac{1}{q}}\right|^q\mr d\mu \approx 0. &
\end{align*}
Thus \eqref{eq:4p41} gives $\omega_q(c_{1,n})\int_{E_n}\phi_n^q\,\mr d\mu \approx c\int_{E_n}\phi_n^q\,\mr d\mu$ and since it is easy to see that $\lim_n \int_{E_n}\phi_n^q\,\mr d\mu > 0$,
since $(\phi_n)$ is extremal, we have that $\lim \omega_q(c_{1,n}) = c = \omega_q\!\left(\frac{(1-q)L^q + qL^{q-1}f}{h}\right)$ or that $c_{1,n} \to \frac{(1-q)L^q + qL^{q-1}f}{h}$, as $n\to \infty$. \\
From \eqref{eq:4p40} we conclude now that $\int_{E_n}(\mc M_{\mc T}\phi_n)^q\,\mr d\,\mu \approx c\int_{E_n}\phi_n^q\,\mr d\mu$ and that, as we have said before we have equality in \eqref{eq:4p39} for the value of $\beta = c^\frac{1}{q}-1$, in the limit.
But \eqref{eq:4p39} comes from \eqref{eq:4p35} and \eqref{eq:4p36} by summing, so we must have equality in \eqref{eq:4p35} in the limit for this value of $\beta = c^\frac{1}{q}-1 = \omega_q\!\left(\frac{(1-q)L^q + qL^{q-1}f}{h}\right)^\frac{1}{q}-1$. \\
But the right side of \eqref{eq:4p35} is minimized for $\beta=\beta_n=\omega_q\!\left(\frac{\ell}{s}\right)^\frac{1}{q}-1$, where $\ell=\lim_n\sum_{I\in\mc A_n} \mu(I)y_{I,n}^q$, $s=\lim_n\int_{\cup\mc A_n}\phi_n^q\,\mr d\mu$. \\
Thus we must have that $\omega_q\!\left(\frac{\ell}{s}\right)=c$, and for the value of $\beta=c^\frac{1}{q}-1$, we get by the equality in \eqref{eq:4p35} that
\[
\lim_n\int_{\cup\mc A_n} (\mc M_{\mc T}\phi_n)^q\,\mr d\mu = c \lim_n \int_{\cup\mc A_n}\phi_n^q\,\mr d\mu.
\]
By this we end the proof of our theorem.
\end{proof}

\noindent We now proceed to prove that
\[
J_n = \int_{E_n} \left| \mc M_{\mc T}\phi_n - c^\frac{1}{q}\phi_n\right|^q\mr d\mu \to 0,\ \ \text{as}\ n \to \infty.
\]
For this proof we are going to use Theorem \ref{thm:4p3} for a sequence $\left(g_{\phi_n}\right)_n$ which arises from $(\phi_n)_n$ in a canonical way and is extremal by construction. We prove the following
\begin{lemma} \label{lem:4p1}
Let $\phi$ be $\mc T$-good and $L \geq f = \int_X\phi\,\mr d\mu$. There exists a measurable function $g_\phi: X\to \mb R^+$ such that for every $I\in \mc T$ such that $I\in S_\phi$ and $\Av_I(\phi)\geq L$ we have that $g_\phi$ assumes two values ($c_I^\phi$ or $0$) on $A(\phi,I)=A_I$.
Moreover $g_\phi$ satisfies $\Av_I(\phi) = \Av_I(g_\phi)$, for every $I\in \mc T$ that contains an element of $S_\phi$ (that is, it is not contained in any of the $A_J$'s).
\end{lemma}

\begin{proof}
Let $\phi$ be $\mc T$-good and $L\geq f$. \\
We first define $g_\phi(t)=\phi(t)$, $t\in X\!\setminus\! E_\phi$, where we set $E_\phi = \{\mc M_{\mc T}\phi \geq L\}$.
Then we write $E_\phi = \cup I_j$, where $I_j$ are maximal elements of the tree $\mc T$ such that $\Av_{I_j}(\phi)\geq L$. Then by Lemma \ref{lem:2p2} we have that $I_j\in S_\phi$. Note now that
\begin{equation} \label{eq:4p43}
I_j = A(\phi,I_j) \cup \bigg( \cup_{\substack{J=I_j\\ J\in S_\phi}} J\bigg),\ \ \text{for any}\ j.
\end{equation}
Fix a $j$. We define first the following function $g_{\phi,j}^{(1)}(t) = \phi(t)$, $t\in I_j\!\setminus\! A(\phi,I_j)$. We write $A_{I_j} = A(\phi,I_j) = \cup_i I_{i,j}$, where $I_{i,j}$ are maximal elements of $\mc T$, subject to the relation $I_{i,j}\subseteq A_{I_j}$ for any $i$. For each $i$ we have $I_{i,j}\in \mc T_{\left(k_{i,j}\right)}$ for some $k_{i,j} \geq 1$, $k_{i,j} \in \mb N$. Let $I_{i,j}'$ be the unique element of $\mc T$, such that $I_{i,j}' \in \mc T_{(k_{i,j}-1)}$ and $I_{i,j}' \supsetneq I_{i,j}$. \\
We set $\Omega_j = \cup_i I_{i,j}'$, for our $j$. Note now that for every $i$, $I_{i,j}' \notin S_\phi$. This is true because of \eqref{eq:4p43} and the structure that the tree $\mc T$ has by its definition.
Consider now a maximal subfamily of $(I_{i,j}')$ that still covers $\Omega_j$. Then we can write $\Omega_j = \cup_{k=1}^\infty I_{i_k,j}'$, for some sequence of integers $i_1\!<\!i_2\!<\! \ldots \!<\!i_k\!<\!\ldots$, possibly finite, where the family $\left(I_{i_k,j}'\right)_k$, $k=1, 2, \ldots$.
Additionally we obviously have that $\Omega_j \subseteq I_j$. By the maximality of any $I_{i_k,j}$, $k\in \mb N$, we have that $I_{i_k,j}' \cap \left(I_j\setminus A_{I_j}\right) \neq \emptyset$, so  there exists $J_j\in S_\phi$ such that $J_j^\star = I_j$ with $J_i \cap I_{i_k,j}' \neq \emptyset$.
Since now each $I_{i_k,j}'$ is not contained in any of $J_i$ (since it contains elements of $A_{I_j}$) we must have that it actually contains any such $J_i$. That is we can write, for any $k\in \mb N$
\[
I_{i_k,j}' = \left[ \cup J_{k,j,m}\right] \cup \left[B_{j,k}\right],
\]
where for any $n\in \mb N$
\[
J_{k,j,m}\in S_\phi,\ \ J_{k,j,m}^\star = I_j\ \ \text{and}\ \ B_{j,k} = I_{i_k,j}'\cap A_{I_j}.
\]
Of course we have $\cup_k B_{j,k} = A_{I_j}$. \\
We define the following function on $I_j$. We name it as $g_{\phi,j,1}: I_j \to \mb R^+$. We set $g_{\phi,j,1}(t) = \phi(t)$, $t\in I_j\setminus A_{I_j}$. Now we are going to construct $g_{\phi,j,1}$ on $A_{I_j}$ in such way that for every $I\in \mc T$ such that $I$ contains an element $J\in S_\phi$, such that $J^\star=I$, we have that $\Av_I(g_{\phi,j,1}) = \Av_I(\phi)$.
We proceed to this as follows: For any $k$, $B_{j,k}$ is a union of elements of the tree $\mc T$. Using Lemma \ref{lem:2p1}, we construct for any $\alpha\in (0,1)$ (that will be fixed later) a pairwise disjoint family of elements of $\mc T$ and subsets of $B_{j,k}$ named as $A_{\phi,j,k}$, such that $\sum_{J\in A_{\phi,j,k}} \mu(J) = \alpha\mu(B_{j,k})$,
We define now the function $g_{\phi,j,1,k}: B_{j,k}\to \mb R^+$ by the following way:
\[
g_{\phi,j,k,1}(t) := \left\{ \begin{aligned}
& c_{j,k,1}^\phi, &t&\in \cup\left\{J: J\in\mc A_{\phi,j,k}\right\} \\
& 0, \quad &t&\in B_{j,k}\setminus \cup\left\{J: J\in\mc A_{\phi,j,k}\right\}
\end{aligned} \right.
\]
such that
\begin{equation} \label{eq:4p44}
\left. \begin{aligned}
\int_{B_{j,k}} g_{\phi,j,k,1}\,\mr d\mu = c_{j,k,1}^\phi \gamma_{j,k,1}^\phi = \int_{B_{j,k}}\phi\,\mr d\mu&,\ \text{and} \\
\int_{B_{j,k}} g_{\phi,j,k,1}^q\,\mr d\mu = \left(c_{j,k,1}^\phi\right)^q \gamma_{j,k,1}^\phi = \int_{B_{j,k}} \phi^q\,\mr d\mu&
\end{aligned}\ \right\},
\end{equation}
where $\gamma_{j,k,1}^\phi = \mu\!\left(\cup_{J\in A_{\phi,j,k}} J\right) = \alpha \mu(B_{j,k})$.
It is easy to see that such choices for $c_{j,k,1}^\phi\geq 0$, $\gamma_{j,k,1}^\phi\in [0,1]$ always exist. In fact we just need to set
\[
\gamma_{j,k,1}^\phi = \left[ \frac{\left(\int_{B_{j,k}} \phi\,\mr d\mu\right)^p}{\int_{B_{j,k}}\phi^p\,\mr d\mu} \right]^\frac{1}{(p-1)} \leq \mu(B_{j,k}),
\]
by H\"{o}lder's inequality, and also $\alpha = \gamma_{j,k,1}^\phi / \mu(B_{j,k})$, $c_{j,k,1}^\phi = \int_{B_{j,k}} \phi\,\mr d\mu / \gamma_{j,k,1}^\phi$. \\
We let then $g_{\phi,j,1}(t) := g_{\phi,j,k,1}(t)$, if $t\in B_{j,k}$. Note that $g_{\phi,j,1}$ may assume more that one positive values on $A_{I_j} = \cup_k B_{j,k}$. It is easy then to see that there exists a common positive value, denoted by $c_{I_j}^{\phi}$ and measurable sets
$L_k\subseteq B_{j,k}$, such that if we define $g'_{\phi,j,1}(t) :=c_{I_j}^{\phi}\chi_{L_k}(t)$ for $t\in B_{j,k}$ for any $k$ and $g'_{\phi,j,1}(t) :=\phi(t)$ for $t\in X\setminus B_{j,k}$, where $\chi_S$ denotes the characteristic function of $S$, we still have
$\int_{B_{j,k}}\phi\,\mr d\mu = \int_{B_{j,k}}g'_{\phi,j,k,1}\,\mr d\mu=c_{I_j}^{\phi}\mu(L_k)$ and
$\int_{A_{I_j}}\phi^q\,\mr d\mu = \int_{A_{I_j}}(g'_{\phi,j,k,1})^q\,\mr d\mu$. For the construction of $L_k$ and $c_{I_j}^{\phi}$
we just need to find first the subsets $L_k$ of $B_{j,k}$, for which the first inequalities mentioned above are true, and this can be done for arbitrary $c_{I_j}^{\phi}$, since $(X,\mu)$ is nonatomic. Then we just need to find the value of the constant $c_{I_j}^{\phi}$ for which the
second integral equality is also true. Note also that for these choices for $L_k$ and $c_{I_j}^{\phi}$ we may not have
$\int_{B_{j,k}}\phi^q\,\mr d\mu = \int_{B_{j,k}}g'_{\phi,j,k,1})^q\,\mr d\mu$, but the respective equality with $A_{I_j}$ in place
of $B_{j,k}$ should be true.
We have thus defined $g_{\phi,j,1}'$. It is obvious now that if $I\in\mc T: I\cap A_{I_j} \neq \emptyset$ and $I\subsetneq A_{I_j}$ (that is $I\cap J\neq \emptyset$ for some $J\in S_\phi$ with $J^\star=I_j$), we must have that $I$ is a union of some of the $I_{i_k,j}'$ and some of the $J$'s for which $J\in S_\phi$ and $J^\star=I$.
Then obviously we should have by the construction we just made that $\int_I g_{\phi,j,1}'\,\mr d\mu = \int_I\phi\,\mr d\mu$. We inductively continue and define $g_{\phi,j,2}':= g_{\phi,j,1}'$ on $A_{I_j}$, and by working also in any of $J\in S_\phi$ such that $J^\star=I_j$, we define it in all the $A_J$'s by the same way as before.
We continue defining with $g_{\phi,j,\ell}'$, $\ell=3, 4, \ldots$. We set at last $g_\phi(t):= \lim_\ell g_{\phi,j,\ell}'(t)$ for any $t\in I_j$. Note that is in fact the sequence $(g_{\phi,j,\ell}'(t))_\ell$ is constant for every $t\in I_j$. Then by it's definition, $g_\phi$ should satisfy the conclusions of our lemma. In this way we derive it's proof.
\end{proof}

\begin{cremark}
It is not difficult to see that for every $I\in S_\phi$, $I\subseteq I_j$ for some $j$ the function $g_\phi$ that is constructed in the previous lemma satisfies $\mu(\{g_\phi=0\} \cap A_I) \geq \mu(\{\phi=0\} \cap A_I)$. This can be seen if we repeat the previous proof by working on the set $\{\phi>0\}\cap A_I$ for any such $I$.
As a consequence, since $E_\phi = \cup_j I_j = \cup_j \bigg(\cup_{\substack{I\subseteq I_j\\ I\in S_\phi}} A(\phi,I)\bigg)$, we conclude that $\mu(\{g_\phi=0\}\cap E_\phi) \geq \mu(\{\phi=0\}\cap E_\phi)$. \\
\end{cremark}

\noindent We prove now the following
\begin{lemma} \label{lem:4p2}
For an extremal sequence $(\phi_n)_n$ of $\mc T$-good functions we have that $\lim_n \mu( \{\phi_n=0\} \cap \{ \mc M_{\mc T}\phi_n \geq L\} ) = 0$.
\end{lemma}

\noindent Before we proceed to the proof of the above Lemma, we prove the following
\begin{lemma} \label{lem:4p3}
For an extremal sequence $(\phi_n)$, consisting of $\mc T$-good functions, such that $S_{\phi_n}$ is the respective subtree of any $\phi_n$, the following holds
\[
\lim_n \sum_{\substack{I\in S_{\phi_n}\\ I\subseteq E_{\phi_n}}} \frac{\left(\int_{A(\phi_n,I)} \phi_n\,\mr d\mu\right)^q}{\alpha_{I,n}^{q-1}} =
\lim_n \int_{E_{\phi_n}} \phi_n^q\,\mr d\mu,
\]
where $\alpha_{I,n} = \mu(A(\phi,I))$, for $I\in S_{\phi_n}$, $n=1, 2, \ldots$.
\end{lemma}

\begin{proof}
Remember that the following inequalities have been used in the evaluation of the function $B_\phi^{\mc T}(f,h,L,1)$:
\begin{equation} \label{eq:4p45}
\int_{I_j} (\mc M_{\mc T}\phi_n)^q\,\mr d\mu \leq \alpha_j \omega_q\! \left(\frac{\beta_j}{\alpha_j}\right).
\end{equation}
Thus we must have equality in the limit in the following inequality:
\begin{equation} \label{eq:4p46}
\sum_j \int_{I_j} (\mc M_{\mc T}\phi_n)^q\,\mr d\mu \leq \sum_j \alpha_j \omega_q \left(\frac{\beta_j}{\alpha_j}\right).
\end{equation}
But in the proof of (4.45) the following inequality was used in order to pass from (3.16) to (3.17) in \cite{4}:
\[
\sum_{\substack{I\in S_{\phi}}} \alpha_I x_I^q \geq \int_X \phi^q\,\mr d\mu.
\]
Now in place of $X$ in the last integral we have the $I_j$'s, so from equality in \eqref{eq:4p46} in the limit, we immediately obtain the statement of our Lemma \ref{lem:4p3}. Our proof is complete.
\end{proof}

\noindent We now return to the
\begin{proof}[Proof of Lemma \ref{lem:4p2}]
It is enough due to the comments mentioned above that $\lim_n \mu(\{g_{\phi_n}=0\} \cap E_{\phi_n}) = 0$. For this, we just need to prove that $$\lim_n \sum_{\substack{I\in S_{\phi_n}\\ I\subseteq E_{\phi_n}}} (\alpha_{I,n}-\gamma_I^{\phi_n}) = 0,$$ where $\alpha_{I,n} = \mu(A(\phi_n,I))$,
and $\gamma_I^{\phi_n}=\mu(A(\phi_n,I)\cap \{g_{\phi_n}>0\})$ for $I\in S_{\phi_n}$, $I\subseteq E_{\phi_n}$. \\
For those $I$ we set
\[
P_{I,n} = \frac{\int_{A_{I,n}} \phi_n^q\,\mr d\mu}{\alpha_{I,n}^{q-1}},\ \text{where}\ A_{I,n} = A(\phi_n,I).
\]
Then we obviously have that $\sum_{\substack{I\in S_{\phi_n}\\ I\subseteq E_{\phi_n}}} \alpha_{I,n}^{q-1} P_{I,n} = \int_{E_{\phi_n}} \phi_n^q$. \\
Additionally $\sum_{\substack{I\in S_{\phi_n}\\ I\subseteq E_{\phi_n}}} (\gamma_I^{\phi_n})^{q-1} P_{I,n} \geq h$, since $0<q<1$ and $\gamma_I^{\phi_n} \leq \alpha_{I,n}$ for $I\in S_{\phi_n}$, $I\subseteq E_{\phi_n}$. However
\begin{multline} \label{eq:4p47}
\sum_{\substack{I\in S_{\phi_n}\\ I\subseteq E_{\phi_n}}} (\gamma_I^{\phi_n})^{q-1} P_{I,n} =
\sum_{\substack{I\in S_{\phi_n}\\ I\subseteq E_{\phi_n}}} (\gamma_I^{\phi_n})^{q-1} \frac{(c_I^{\phi_n})^q \gamma_I^{\phi_n}}{(\alpha_{I,n})^{q-1}} = \\
\sum_{\substack{I\in S_{\phi_n}\\ I\subseteq E_{\phi_n}}} \frac{(\gamma_I^{\phi_n} . c_I^{\phi_n})^q}{(\alpha_{I,n})^{q-1}} =
\sum_{\substack{I\in S_{\phi_n}\\ I\subseteq E_{\phi_n}}} \frac{\left(\int_{A_{I,n}} \phi_n\,\mr d\mu\right)^q}{(\alpha_{I,n})^{q-1}} \approx
\int_{E_{\phi_n}} \phi_n^q\,\mr d\mu,
\end{multline}
by Lemmas \ref{lem:4p1} and \ref{lem:4p3}. \\
We define now for any $R>0$ the set
\[
S_{\phi_n,R} = \cup\left\{ A_{I,n}: I\in S_{\phi_n},\ I\subseteq E_{\phi_n},\ P_{I,n}<R(a_{I,n})^{2-q}\right\}.
\]
Then for $I\in S_{\phi_n}$ such that $I\subseteq E_{\phi_n}$ and $P_{I,n} < R(\alpha_{I,n})^{2-q}$ we have that
\begin{align} \label{eq:4p48}
& \int_{A_{I,n}} \phi_n^q\,\mr d\mu < R \alpha_{I,n} \implies\ \text{(by summing up to all such $I$)} \notag \\
& \int_{S_{\phi_n,R}} \phi_n^q\,\mr d\mu < R \mu(S_{\phi_n,R}).
\end{align}
Additionally we have that
\begin{equation} \label{eq:4p49}
\Bigg|\!\!\! \sum_{\substack{I\in S_{\phi_n}\\ I\subseteq E_{\phi_n},\ P_{I,n}\geq R\alpha_{I,n}^{2-q}}}\hspace{-10pt} \alpha_{I,n}^{q-1} P_{I,n} - \int_{E_{\phi_n}} \phi_n^q\,\mr d\mu\ \Bigg| =
\int_{S_{\phi_n,R}} \phi_n^q\,\mr d\mu,
\end{equation}
and
\begin{align} \label{eq:4p50}
& \Bigg|\!\! \sum_{\substack{I\in S_{\phi_n}\\ I\subseteq E_{\phi_n},\ P_{I,n}\geq R\alpha_{I,n}^{2-q}}}\hspace{-10pt} \left(\gamma_I^{\phi_n}\right)^{q-1} P_{I,n} - \int_{E_{\phi_n}} \phi_n^q\,\mr d\mu\ \Bigg| \overset{\eqref{eq:4p47}}{\approx} \notag \\
& \Bigg|\!\! \sum_{\substack{I\in S_{\phi_n}\\ I\subseteq E_{\phi_n},\ P_{I,n}\geq R\alpha_{I,n}^{2-q}}}\hspace{-10pt} \left(\gamma_I^{\phi_n}\right)^{q-1} P_{I,n} - \sum_{\substack{I\in S_{\phi_n}\\ I\subseteq E_{\phi_n}}} \left(\gamma_I^{\phi_n}\right)^{1-q} P_{I,n}\ \Bigg| = \notag \\
& \sum_{\substack{I\in S_{\phi_n}\\ I\subseteq E_{\phi_n},\ P_{I,n}< R\alpha_{I,n}^{2-q}}}\hspace{-10pt} \left(\gamma_I^{\phi_n}\right)^{q-1} P_{I,n} =
\sum_{\substack{I\in S_{\phi_n}\\ I\subseteq E_{\phi_n},\ P_{I,n}< R\alpha_{I,n}^{2-q}}}\hspace{-10pt} \left(\gamma_I^{\phi_n}\right)^{q-1} \frac{(c_I^{\phi_n})^q\gamma_I^{\phi_n}}{(\alpha_{I,n})^{q-1}} = \notag \\
& \sum_{\substack{I\in S_{\phi_n}\\ I\subseteq E_{\phi_n},\ P_{I,n}< R\alpha_{I,n}^{2-q}}}\hspace{-10pt} \frac{(\gamma_I^{\phi_n} c_I^{\phi_n})^q}{(\alpha_{I,n})^{q-1}} =
\sum_{\substack{I\in S_{\phi_n}\\ I\subseteq E_{\phi_n},\ P_{I,n}< R\alpha_{I,n}^{2-q}}}\hspace{-10pt} \frac{\left(\int_{A_{I,n}} \phi_n\,\mr d\mu\right)^q}{\alpha_{I,n}^{2-q}} \approx
\int_{S_{\phi_n,R}} \phi_n^q\,\mr d\mu,
\end{align}
where the last equality in the limit is explained by the same reasons as Lemma \ref{lem:4p3} does.
Using \eqref{eq:4p49} and \eqref{eq:4p50} we conclude that
\begin{equation} \label{eq:4p51}
\limsup_n\hspace{-10pt} \sum_{\substack{I\in S_{\phi_n}\\ I\subseteq E_{\phi_n},\ P_{I,n}\geq R\alpha_{I,n}^{2-q}}}\hspace{-10pt} \left[ \left(\gamma_I^{\phi_n}\right)^{q-1} - (\alpha_{I,n})^{q-1} \right] P_{I,n} \leq
2 \lim_n \int_{S_{\phi_n,R}} \phi_n^q\,\mr d\mu,
\end{equation}
%
%
By Theorem \ref{thm:4p3} now, and Lemma \ref{lem:2p3} (using the form that the $A_{I,n}$ have, and a diagonal argument) we have that the following is true
\begin{equation} \label{eq:4p52}
\lim_n \int_{S_{\phi_n,R}} (\mc M_{\mc T}\phi_n)^q\,\mr d\mu =
c \lim_n \int_{S_{\phi_n,R}} \phi_n^q\,\mr d\mu.
\end{equation}
Since $\mc M_{\mc T}\phi \geq f$, on $X$ we conclude by \eqref{eq:4p48} and \eqref{eq:4p52} that
\begin{equation} \label{eq:4p53}
f^q \limsup_n \mu(S_{\phi_n,R}) \leq c\,R \limsup \mu(S_{\phi_n,R}).
\end{equation}
Thus if $R>0$ is chosen small enough, we must have because of \eqref{eq:4p53} that $\lim_n \mu(S_{\phi_n,R})=0$, thus by \eqref{eq:4p48} we have $\lim_n \int_{S_{\phi_n,R}}\phi^q\,\mr d\mu=0$, and so by \eqref{eq:4p51} we obtain
\begin{equation} \label{eq:4p54}
\lim_n\hspace{-10pt} \sum_{\substack{I\in S_{\phi_n}\\ I\subseteq E_{\phi_n}, P_{I,n} \geq R a_{I,n}^{2-q}}}\hspace{-10pt} \left[ \left(\gamma_I^{\phi_n}\right)^{q-1} - \left(\alpha_{I,\phi_n}\right)^{q-1}\right] P_{I,n} = 0.
\end{equation}
%
%
%
%
We consider now, for any $y>0$ the function $\phi_y(x) = \frac{x^{q-1}y^{2-q}-y}{y-x}$, defined for $x\in(0,y)$. Is is easy to see that $\lim_{x\to 0^+} \phi_y(x) = +\infty$, $\lim_{x\to y^-} \phi_y(x)=1-q$. Moreover $\phi_y'(x) = \frac{(y-1)x^{q-2}y^{3-q} - (q-2)_x^{q-1}y^{2-q} - y}{(y-x)^2}$, $x\in (0,y)$. \\
Then by setting $x=\lambda y$, $\lambda\in(0,1)$ we define the following function $g(\lambda)=(q-1)\lambda^{q-2} - (q-2)\lambda^{q-1} - 1$, which as is easily seen satisfies $g(\lambda)<0$, for all $\lambda\in(0,1)$. But $\phi_y'(x) = \frac{y\,g(\lambda)}{(1-\lambda)^2y^2}<0$, so that $\phi_y$ is decreasing on $(0,y)$.
Thus $\phi_y(x) \geq 1-q$, for all $x\in(0,y)$ $\implies x^{q-1}y^{2-q}-y \geq (1-q)(y-x)$, $\forall x\in(0,y)$. \\
From the above and \eqref{eq:4p54} we see that
\[
\lim_n \quad \sum_{\mathclap{\substack{I\in S_{\phi_n}\\ I\subseteq E_{\phi_n}, P_{I,n}\geq R_{\alpha_{I,n}}^{2-q}}}} \quad  \left(\alpha_{I,n}-\gamma_{I,n}^{\phi_n}\right) = 0 \implies
\mu(E_{\phi_n}) - \mu(S_{\phi_n,R}) - \quad\sum_{\mathclap{\substack{I\in S_{\phi_n}\\ I\subseteq E_{\phi_n}, P_{I,n}\geq R_{\alpha_{I,n}}^{2-q}}}}\quad \left(\gamma_n^{\phi_n}\right) \approx 0.
\]
Since then $\mu(S_{\phi_n,R})\to0$ we conclude that
\begin{equation} \label{eq:4p55}
\mu(E_{\phi_n}) \approx \quad\sum_{\mathclap{\substack{I\in S_{\phi_n}\\ I\subseteq E_{\phi_n}, P_{I,n}\geq R_{\alpha_{I,n}}^{2-q}}}}\quad  \gamma_I^{\phi_n} \leq
\sum_{I\in S_{\phi_n}, I\subseteq E_{\phi_n}} (\gamma_I^{\phi_n}) \leq
\sum_{I\in S_{\phi_n}, I\subseteq E_{\phi_n}} \alpha_{I,n} =
\mu(E_{\phi_n}).
\end{equation}
Thus from \eqref{eq:4p55} we immediately see that $\sum_{I\in S_{\phi_n}, I\subseteq E_{\phi_n}}\!\! \left(\alpha_{I,n}-\gamma_I^{\phi_n}\right) \approx 0$, or that $\mu(\{g_{\phi_n}=0\}\cap E_{\phi_n}) \approx 0$, and by this we end the proof of Lemma \ref{lem:4p2}.
\end{proof}
\medskip

Now as we have mentioned before, by the construction of $g_{\phi_n}$, we have that $\int_I g_{\phi_n}\,\mr d\mu = \int_I \phi_n\,\mr d\mu$, for every $I\in S_{\phi_n}$.

Thus $\mc M_{\mc T}g_\phi \geq \mc M_{\mc T}\phi\ \text{on}\ X \implies \lim_n \int_X (\mc M_{\mc T}g_{\phi_n})^q\,\mr d\mu \geq h\, c$. Since $\int_X g_{\phi_n}\,\mr d\mu = f$ and $\int g_{\phi_n}^q\,\mr d\mu=h$, by construction, we conclude that $$\lim_n\int _X (\mc M_{\mc T}g_{\phi_n})^q\,\mr d\mu =ch,$$ or that $(g_{\phi_n})_n$ is an extremal sequence.

We prove now the following Lemmas needed for the end of the proof of the characterization of the extremal sequences for \eqref{eq:1p8}

\begin{lemma} \label{lem:4p4}
With the above notation there holds:
\[
\lim \int_{E_{\phi_n}} \left|\mc M_{\mc T}g_{\phi_n} - c^\frac{1}{q}g_{\phi_n}\right|^q\mr d\mu = 0.
\]
\end{lemma}

\begin{proof}
We define for every $n\in\mb N^\star$ the set:
\[
\Delta_n = \left\{ t\in E_{\phi_n}:\ \mc M_{\mc T}g_{\phi_n} \geq c^\frac{1}{q} g_{\phi_n}(t)\right\}.
\]
It is obvious, by passing, if necessary to a subsequence that
\begin{equation} \label{eq:4p56}
\lim_n \int_{\Delta_n} (\mc M_{\mc T}g_{\phi_n})^q\,\mr d\mu \geq c \lim_n \int_{\Delta_n} g_{\phi_n}^q\,\mr d\mu.
\end{equation}
We consider now for every $I\in S_{\phi_n}$, $I\subseteq E_{\phi_n}$ the set $(E_{\phi_n}\!\setminus \Delta_n)\cap A_{I,n}$ where $A_{I,n} = A(\phi_n,I)$. We distinguish two cases.
\begin{enumerate}[\hspace{-5pt}(i)]
\item $\Av_I(\phi_n) = y_{I,n} > c^\frac{1}{q} c_I^{\phi_n}$, where $c_I^{\phi_n}$ is the positive value of $g_{\phi_n}$ on $A_{I,n}$ (if it exists). Then because of Lemma \ref{lem:4p1} we have that
\[
\mc M_{\mc T}g_{\phi_n}(t) \geq \Av_I(g_{\phi_n}) =\Av_I(\phi_n) > c^\frac{1}{q} c_I^{\phi_n} \geq c^\frac{1}{q} g_\phi(t),
\]
for each $t\in A_{I,n}$. Thus $(E_{\phi_n}\!\setminus \Delta_n)\cap A_{I,n} = \emptyset$ in this case. \\
We study now the second one. \\
\item $y_I \leq c^\frac{1}{q} c_I^{\phi_n}$. Let now $t\in A_{I,n}$ with $g_\phi(t)>0$, that is $g_{\phi_n}(t) = c_I^{\phi_n}$. We prove that for this $t$ we have $\mc M_{\mc T}g_{\phi_n} \leq c^\frac{1}{q} g_{\phi_n}(t) = c^\frac{1}{q} c_I^{\phi_n}$.
\end{enumerate}
Suppose now that we have the opposite inequality. Then there exists $J_t\in \mc T$ such that $t\in J_t$ and $\Av_{J_t}(g_{\phi_n}) > c^\frac{1}{q} c_I^{\phi_n}$. Then one of the following holds
\begin{enumerate}[(a)]
\item $J_t \subseteq A_{I,n}$. Then by the form of $g_{\phi_n} / A_{I,n}$ (equals $0$ or $c_i^{\phi_n}$), we have that $\Av_{J_t}(g_{\phi_n}) \leq c_I^{\phi_n} \leq c^\frac{1}{q} c_I^{\phi_n}$, which is a contradiction. Thus this case is excluded.
\item $J_t$ is not a subset of $A_{I,n}$. Then two subcases can occur.
\begin{enumerate}[($\mr b_1$)]
\setlength{\parsep}{0pt}
\item $J_t \subseteq I \subseteq E_{\phi_n}$ and contains properly an element of $S_{\phi_n}$, $J'$, for which $(J')^\star = I$. Since now (\rnum{2}) holds, $t\in J_t$ and $\Av_{J_t}(g_{\phi_n}) > c^\frac{1}{q}c_I^{\phi_n}$, we must have that $J' \subsetneq J_t \subsetneq I$.
We choose now an element $J_t'$ of $\mc T$, $J_t \subsetneq I$ which contains $J_t$, with maximum value on the average $\Av_{J_t'}(\phi_n)$. Then by it's choice we have that for each $K\in \mc T$ such that $J_t' \subset K \subsetneq I$ there holds $\Av_K(\phi) \leq \Av_{J_t'}(\phi)$. Since now $I\in S_{\phi_n}$ and $\Av_I(\phi_n) \leq c^\frac{1}{q}c_I^{\phi_n}$, by Lemma \ref{lem:2p2} and the choice of $J_t'$ we have that $\Av_K(\phi_n) < \Av_{J_t'}(\phi_n)$ for every $K\in \mc T$ such that $J_t' \subseteq K$. So again by Lemma \ref{lem:2p2} we conclude that $J_t'\in S_{\phi_n}$. But this is impossible, since $J' \subsetneq J_t' \subsetneq I$, $J'\!, I\in S_{\phi_n}$ and $(J')^\star = I$. We turn now to the second subcase.
\item $I\subsetneq J_t$. Then by application of Lemma \ref{lem:4p1} we have that $\Av_{J_t}(\phi_n) = \Av_{J_t}(g_{\phi_n}) > c^\frac{1}{q}c_I^{\phi_n} \geq y_{I,n} = \Av_I(\phi_n)$ which is impossible by Lemma \ref{lem:2p2}, since $I\in S_{\phi_n}$.
\end{enumerate}
\end{enumerate}
Thus in any of the two cases ($b_1$) and ($b_2$) we have proved that we have \\ $(E_{\phi_n}\!\setminus \Delta_n) \cap A_{I,n} = A_{I,n} \setminus \{g_{\phi_n}\!=0\}$, while we showed that in the case (\rnum{1}), $(E_{\phi_n}\!\setminus \Delta_n)\cap A_{I,n} = \emptyset$. \\
Since $\cup \{A_I: I\in S_{\phi_n},\ I\subseteq E_{\phi_n}\} \approx E_{\phi_n}$, we conclude by the above discussion that $E_{\phi_n}\!\setminus \Delta_n$ can be written as $\left(\cup_{I\in S_{1,\phi_n}} A_{I,n}\right)\setminus \Gamma_{\phi_n}$, where $\mu(\Gamma_{\phi_n})\to 0$ and $S_{1,\phi_n}$ is a subtree of $S_{\phi_n}$. Then by Lemma \ref{lem:2p3} and Theorem \ref{thm:4p3}, by passing if necessary to a subsequence, we have
\[
\lim_n \int_{\cup_{I\in S_{1,\phi_n}}\! A_{I,n}} (\mc M_{\mc T}\phi_n)^q\,\mr d\mu =
c \lim_n \int_{\cup_{I\in S_{1,\phi_n}}\! A_{I,n}} \phi_n^q\,\mr d\mu,
\]
so since $\lim_n \mu(\Gamma_{\phi_n}) = 0$, we conclude that
\begin{equation} \label{eq:4p57}
\lim_n \int_{E_{\phi_n}\!\setminus\Delta_n} (\mc M_{\mc T}\phi_n)^q\,\mr d\mu =
c \lim_n \int_{E_{\phi_n}\!\setminus\Delta_n} \phi_n^q\,\mr d\mu.
\end{equation}
Because then of the relation $\mc M_{\mc T}g_\phi \geq \mc M_{\mc T}\phi$, on $X$ we have by \eqref{eq:4p57} as a consequence that
\begin{equation} \label{eq:4p58}
\lim_n \int_{E_{\phi_n}\!\setminus\Delta_n} (\mc M_{\mc T}g_{\phi_n})^q\,\mr d\mu \geq
c \lim_n \int_{E_{\phi_n}\!\setminus \Delta_n} g_{\phi_n}^q\,\mr d\mu.
\end{equation}
Adding \eqref{eq:4p56} and \eqref{eq:4p58}, we obtain
\begin{equation} \label{eq:4p59}
\lim_n \int_{E_{\phi_n}} (\mc M_{\mc T}g_{\phi_n})^q\,\mr d\mu \geq
c\int_{E_{\phi_n}} g_{\phi_n}^q\,\mr d\mu,
\end{equation}
which in fact is an equality, because if we had strict inequality in \eqref{eq:4p59} we would produce since $g_{\phi_n} = \phi_n$ on $X\setminus E_{\phi_n}$, that $\lim_n \int_X (\mc M_{\mc T}g_{\phi_n})^q\,\mr d\mu > c\, h$, as we can easily see.
This is a contradiction, since $\int_X g_{\phi_n}\,\mr d\mu = f$ and $\int_X g_{\phi_n}^q\,\mr d\mu = h$, for every $n\in\mb N$, and because of Theorem \ref{thm:1}.
Thus we must have equality in both \eqref{eq:4p56} and \eqref{eq:4p58}. \\
Our proof is completed.
\end{proof}
We proceed now to the following

\begin{lemma} \label{lem:4p5}
Let $X_n\subset X$, and $h_n, z_n: X_n\to \mb R^+$ be measurable functions such that $h_n^q = z_n$, where $q\in(0,1)$ is fixed. Suppose additionally that $g_n, w_n: X\to \mb R^+$ satisfy $g_n^q = w_n$. Suppose also that $g_n \geq h_n$, on $X_n$.
Then if $\lim_n \int_{X_n} (w_n-z_n)\,\mr d\mu = 0$ and the sequence $\int_{X_n}\! w_n\,\mr d\mu$ is bounded, we have that $\lim_n \int_{X_n} (g_n-h_n)^q\,\mr d\mu = 0$.
\end{lemma}

\begin{proof}
We set $I_n = \int_{x_n} \Big(w_n^\frac{1}{q} - z_n^\frac{1}{q}\Big)^q\mr d\mu$. \\
For every $p>1$, the following elementary inequality is true $x^p-y^p \leq p(x-y)x^{p-1}$, for $x>y>0$. Thus for $p=\frac{1}{q}$, we have $w_n^p - z_n^p \leq p(w_n-z_n)w_n^{p-1} \implies$
\begin{equation} \label{eq:4p60}
I_n \leq \left(\frac{1}{q}\right)^q \int_{X_n} (w_n-z_n)^q w_n^{1-q}\,\mr d\mu.
\end{equation}
If we use now H\"{o}lder's inequality in \eqref{eq:4p60} we immediately obtain that $I_n \leq \left(\frac{1}{q}\right)^q \left(\int_{X_n} (w_n-w)\,\mr d\mu\right)^q \left(\int_{X_n} w_n\right)^{1-q} \to 0$, as $n\to \infty$, by our hypothesis.
\end{proof}

\noindent Let now $(\phi_n)$ be an extremal sequence of functions. We define $g_{\phi_n}': (X,\mu)\to \mb R^+$ by
\[
g_{\phi_n}'(t) = c_I^{\phi_n},\ t\in A_{I,n} = A(\phi_n,I),\ \ \text{for}\ I\in S_{\phi_n}.
\]
\noindent We prove now the following
\begin{lemma} \label{lem:4p6}
With the above notation $\lim_n \int_{E_{\phi_n}} |g_{\phi_n}' - \phi_n|^q\,\mr d\mu = 0$.
\end{lemma}

\begin{proof}
We are going to use again the inequality $t + \frac{1-q}{q} \geq \frac{t^q}{q}$, which holds for every $t>0$ and $q\in(0,1)$. In view of Lemma \ref{lem:4p5} we just need to prove that
\[
\int_{\{\phi_n \geq g_{\phi_n}'\}\cap E_{\phi_n}}\!\!\! [\phi_n^q - (g_{\phi_n}')^q]\,\mr d\mu \to 0\quad \text{and}\quad
\int_{\{g_{\phi_n}' > \phi_n\}\cap E_{\phi_n}}\!\!\! [(g_{\phi_n}')^q - \phi_n]\,\mr d\mu \to 0.
\]
We proceed to this as follows. \\
For every $I\in S_{\phi_n}$, $I\subseteq E_{\phi_n}$, we set
\begin{align*}
\Delta_{I,n}^{(1)} = \{g_{\phi_n}' \leq \phi_n\} \cap A(\phi_n, I), \\
\Delta_{I,n}^{(2)} = \{\phi_n < g_{\phi_n}'\} \cap A(\phi_n,I).
\end{align*}
From the inequality mentioned in the beginning of this proof we have that, if $c_I^{\phi_n} > 0$, then
\[
\frac{\phi_n(x)}{c_I^{\phi_n}} + \frac{1-q}{q} \geq \frac{1}{q} \frac{\phi_n^q(x)}{(c_I^{\phi_n})^q}, \ \ \forall x\in A_{I,n},
\]
so integrating over every $\Delta_{I,n}^{(j)}$, $j=1,2$ we obtain
\begin{align} \label{eq:4p61}
& \frac{1}{c_I^{\phi_n}} \int_{\Delta_{I,n}^{(j)}} \phi_n\,\mr d\mu + \frac{1-q}{q}\mu(\Delta_{I,n}^{(j)}) \geq
\frac{1}{q} \frac{1}{(c_I^{\phi_n})^q} \int_{\Delta_{I,n}^{(j)}} \phi_n^q\,\mr d\mu \implies \notag \\
& \sum_{I\in S_{\phi_n}'} (c_I^{\phi_n})^q \int_{\Delta_{I,n}^{(j)}} \phi_n\,\mr d\mu + \frac{1-q}{q} \sum_{I\in S_{\phi_n}'} \mu(\Delta_{I,n}^{(j)}) (c_I^{\phi_n})^q \geq
 \frac{1}{q} \int_{\cup_{I\in S_{\phi_n}'}\!\!\! \Delta_{I,n}^{(j)}} \phi_n^q\,\mr d\mu,
\end{align}
for $j=1,2$, where $S_{\phi_n}' = \{I\in S_{\phi_n}: I\subseteq E_{\phi_n},\ c_I^{\phi_n}>0\}$. \\
From the definition of $g_{\phi_n}'$ we see that \eqref{eq:4p61} gives
\begin{multline} \label{eq:4p62}
\int_{\cup_{I\in S_{\phi_n}'} \Delta_{I,n}^{(j)}} (g_{\phi_n}')^{q-1}\phi_n\,\mr d\mu + \frac{1-q}{q} \sum_{I\in S_{\phi_n}'} (c_I^{\phi_n})^q \mu(\Delta_{I,n}^{(j)}) \geq \\
\frac{1}{q} \int_{\cup_{I\in S_{\phi_n}'} \Delta_{I,n}^{(j)}} \phi_n^q\,\mr d\mu,\ \ \text{for}\ j=1,2.
\end{multline}
Note now that
\[
\sum_{I\in S_{\phi_n}'} (c_I^{\phi_n})^q \mu(\Delta_{I,n}^{(j)}) = \begin{cases}
\int_{\{\phi_n\geq g_{\phi_n}'\}\cap E_{\phi_n}} (g_{\phi_n}')^q\,\mr d\mu, & j=1 \\[10pt]
\int_{\{\phi_n<g_{\phi_n}'\}\cap E_{\phi_n}} (g_{\phi_n}')^q\,\mr d\mu, & j=2,
\end{cases}
\]
and
\[
\int_{\cup_{I\in S_{\phi_n}'}\!\!\!\Delta_{I,n}^{(j)}} \phi_n^q\,\mr d\mu =  \begin{cases}
\int_{\{\phi_n\geq g_{\phi_n}'\}\cap E_{\phi_n}} \phi_n^q\,\mr d\mu, & j=1 \\[10pt]
\int_{\{\phi_n<g_{\phi_n}'\}\cap E_{\phi_n}} \phi_n^q\,\mr d\mu, & j=2,
\end{cases}
\]
because if $c_I^{\phi_n}=0$, for some $I\in S_{\phi_n}'$, $I\subseteq E_{\phi_n}$, then $\phi_n=0$ on the respective $A_{I,n}$, and conversely. Additionally:
\[
\int_{\cup_{I\in S_{\phi_n}'}\!\!\!\Delta_{I,n}^{(j)}} (g_{\phi_n}')^{q-1}\phi_n\,\mr d\mu =  \begin{cases}
\int_{\{\phi_n\geq g_{\phi_n}'\}\cap E_{\phi_n}} (g_{\phi_n}')^{q-1}\phi_n\,\mr d\mu, & j=1 \\[10pt]
\int_{\{\phi_n<g_{\phi_n}'\}\cap E_{\phi_n}} (g_{\phi_n}')^{q-1}\phi_n\,\mr d\mu, & j=2.
\end{cases}
\]
So we conclude the following two inequalities:
\begin{equation} \label{eq:4p63}
\int\limits_{\{0 < g_{\phi_n}' \leq \phi_n\} \cap E_{\phi_n}}\hspace{-26pt} (g_{\phi_n}')^{q-1}\phi_n\,\mr d\mu +
\frac{1-q}{q}\ \int\limits_{\mathclap{\{g_{\phi_n}' \leq \phi_n\} \cap E_{\phi_n}}} (g_{\phi_n}')^q\,\mr d\mu\quad \geq \quad
\frac{1}{q}\ \int\limits_{\mathclap{\{g_{\phi_n}' \leq \phi_n\} \cap E_{\phi_n}}} \phi_n^q\,\mr d\mu,
\end{equation}
and
\begin{equation} \label{eq:4p64}
\int\limits_{\{g_{\phi_n}' > \phi_n\} \cap E_{\phi_n}}\hspace{-21pt} (g_{\phi_n}')^{q-1}\phi_n\,\mr d\mu +
\frac{1-q}{q}\ \int\limits_{\mathclap{\{g_{\phi_n}' > \phi_n\} \cap E_{\phi_n}}} (g_{\phi_n}')^q\,\mr d\mu\quad \geq \quad
\frac{1}{q}\ \int\limits_{\mathclap{\{g_{\phi_n}' >\phi_n\} \cap E_{\phi_n}}} \phi_n^q\,\mr d\mu.
\end{equation}
If we sum the above inequalities we get:
\begin{equation} \label{eq:4p65}
\sum_{I\in S_{\phi_n}'} (c_I^{\phi_n})^{q-1} (c_I^{\phi_n} \gamma_I^{\phi_n}) + \frac{1-q}{q} \int_{E_{\phi_n}} (g_{\phi_n}')^q\,\mr d\mu \geq
\frac{1}{q} \int_{E_{\phi_n}} \phi_n^q\,\mr d\mu.
\end{equation}
Now the following are true because of Lemma \ref{lem:4p2}
\[
\int_{E_{\phi_n}} \phi_n^q\,\mr d\mu = \sum_{I\in S_{\phi_n}'} \gamma_I^{\phi_n}(c_I^{\phi_n})^q \approx
\int_{E_{\phi_n}} (g_{\phi_n}')^q\,\mr d\mu.
\]
Thus in \eqref{eq:4p65} we must have euality in the limit. As a result we obtain equalities in both \eqref{eq:4p63} and \eqref{eq:4p64} in the limit. \\
As a consequence, if we set
\[
t_n = \int_{\{g_{\phi_n}'\leq \phi_n\}\cap E_{\phi_n}} \phi_n^q\,\mr d\mu, \qquad
s_n = \int_{\{g_{\phi_n}'\leq \phi_n\}\cap E_{\phi_n}} (g_{\phi_n}')^q\,\mr d\mu,
\]
we must have that
\begin{equation} \label{eq:4p66}
\int_{\{\phi_n \geq g_{\phi_n}' > 0\}\cap E_{\phi_n}} \phi_n (g_{\phi_n}')^q\,\mr d\mu + \frac{1-q}{q}s_n \approx \frac{1}{q}t_n.
\end{equation}

But as can be easily seen we have that
\begin{equation} \label{eq:4p67}
\bigg[\qquad \int\limits_{\qquad\mathclap{\{0 < g_{\phi_n}' \leq \phi_n\}\cap E_{\phi_n}}} \phi_n(g_{\phi_n}')^{q-1}\,\mr d\mu\bigg]^q \cdot
\bigg[\qquad \int\limits_{\qquad\mathclap{\{0 < g_{\phi_n}' \leq \phi_n\}\cap E_{\phi_n}}} (g_{\phi_n}')^q\,\mr d\mu\bigg]^{1-q} \geq
\qquad \int\limits_{\qquad\mathclap{\{0 < g_{\phi_n}' \leq \phi_n\}\cap E_{\phi_n}}} \phi_n^q\,\mr d\mu.
\end{equation}
From \eqref{eq:4p66} and \eqref{eq:4p67} we have as a result that
\begin{equation} \label{eq:4p68}
\frac{t_n^\frac{1}{q}}{s_n^{\frac{1}{q}-1}} + \frac{1-q}{q}s_n \leq \frac{1}{q}t_n \implies
\left(\frac{t_n}{s_n}\right)^\frac{1}{q} + \frac{1-q}{q} \leq \frac{1}{q}\left(\frac{t_n}{s_n}\right),
\end{equation}
in the limit. \\
But for every $n\in\mb N$ we have that $\left(\frac{t_n}{s_n}\right)^\frac{1}{q} + \frac{1-q}{q} \geq \frac{1}{q}\left(\frac{t_n}{s_n}\right)$. Thus we have equality in \eqref{eq:4p68} in the limit. This means that $\frac{t_n}{s_n}\approx 1$ and since $(t_n)_n$ and $(s_n)_n$ are bounded sequences, we conclude that
\[
t_n-s_n\to 0 \implies \int_{\{g_{\phi_n}'\leq \phi_n\}\cap E_{\phi_n}} [\phi_n^q-(g_{\phi_n}')^q]\,\mr d\mu \to 0,\ \ \text{as}\ n\to\infty.
\]
In a similar way we prove that $\int_{\{\phi_n < g_{\phi_n}'\}\cap E_{\phi_n}} [(g_{\phi_n}')^q - \phi_n^q]\,\mr d\mu \to 0$. Thus Lemma \ref{lem:4p6} is proved. \\
\end{proof}

\noindent We now proceed to the following.
\begin{lemma} \label{lem:4p7}
With the above notation, we have that
\[
\lim_n \int_{E_{\phi_n}} \left| \mc M_{\mc T}\phi_n - c^\frac{1}{q}\phi_n\right|^q\mr d\mu = 0.
\]
\end{lemma}

\begin{proof}
We set $J_n = \int_{E_{\phi_n}} \left|\mc M_{\mc T}\phi_n - c^\frac{1}{q}\phi_n\right|^q\mr d\mu$. \\
It is true that $(x+y)^q < x^q+y^q$, whenever $x,y>0$, $q\in (0,1)$.
Thus
\begin{multline*}
J_n \leq \int_{E_{\phi_n}} |\mc M_{\mc T}\phi_n - \mc M_{\mc T}g_{\phi_n}|^q\,\mr d\mu + \int_{E_{\phi_n}} |\mc M_{\mc T}g_{\phi_n}-c^\frac{1}{q}g_{\phi_n}|^q\,\mr d\mu + \\
c\int_{E_{\phi_n}} |g_{\phi_n}-\phi_n|^q\,\mr d\mu = J_n^{(1)} + J_n^{(2)} + J_n^{(3)}.
\end{multline*}
By Lemmas \ref{lem:4p6} and \ref{lem:4p2} we have that $J_n^{(3)}\to 0$, as $n\to \infty$. Also, $J_n^{(2)}\to 0$ by Lemma \ref{lem:4p4}. We look now at $J_n^{(1)} = \int_{E_{\phi_n}} |\mc M_{\mc T}\phi_n - \mc M_{\mc T}g_{\phi_n}|^q\,\mr d\mu$.
As we have mentioned before $\mc M_{\mc T}g_{\phi_n} \geq \mc M_{\mc T}\phi_n$, on $X$, thus $J_n^{(1)} = \int_{E_{\phi_n}} (\mc M_{\mc T}g_{\phi_n} - \mc M_{\mc T}\phi_n)^q\,\mr d\mu$. \\
Since $\lim_n\int_{E_{\phi_n}} (\mc M_{\mc T}\phi_n)^q\,\mr d\mu = \lim_n\int_{E_{\phi_n}} (\mc M_{\mc T}\phi_m)^q\,\mr d\mu = c \lim \int_{E_{\phi_n}} \phi_n^q$ we immediately see that $J_n^{(1)}\to 0$, by Lemma \ref{lem:4p5}. \\
The proof of Lemma \ref{lem:4p7} is thus complete, completing also the proof of Theorem \ref{thm:a}.
\end{proof}

\begin{remark} \label{rem:4p1}
We need to mention that Theorem \ref{thm:a} holds true on $\mb R^n$ without the hypothesis that the sequence $(\phi_n)_n$ consists of $\mc T$-good functions. This is true since in the case of $\mb R^n$, where $\mc T$ is the usual tree of dyadic subcubes of a fixed cube $Q$, the class of $\mc T$-good functions contains the one of the dyadic step functions on $Q$, which are dense on $L^1(X,\mu)$.
\end{remark}

\vspace{50pt}
\noindent Nikolidakis Eleftherios\\
Visiting Professor\\
Department of Mathematics \\
University of Ioannina \\
Greece\\
E-mail address: lefteris@math.uoc.gr

\end{document}